\documentclass[12pt]{article}
\usepackage[utf8]{inputenc}
\usepackage[margin=0.5in]{geometry}
\usepackage{amsmath}
\usepackage{amssymb}
\numberwithin{equation}{section}
\usepackage[english]{babel}
\usepackage{amsthm}
\usepackage{bbm}
\usepackage{amsfonts}
\usepackage{hyperref}
\usepackage{graphicx}
\usepackage{float}

\usepackage[format=plain,
            font=it]{caption}

\newtheorem{theorem}{Theorem}[section]
\newtheorem{corollary}{Corollary}[theorem]
\newtheorem{lemma}[theorem]{Lemma}
\newtheorem{proposition}[theorem]{Proposition}
\newtheorem*{remark}{Remark}
\newtheorem{definition}[theorem]{Definition}

\newcommand{\upd}{\mathrm{d}}  

\def\a{\alpha }     \def\b{\beta  }         \def\g{\gamma }
     \def\d{\delta}          \def\D{\Delta }
    \def\ve{\varepsilon}    
\def\h{\eta }               
          \def\l{\lambda }    
    \def\m{\mu}             \def\n{\nu}         
                     
                \def\u{\upsilon }

       \def\w{\omega }         \def\W{\Omega }

\title{Asymptotic analysis of subwavelength halide perovskite resonators\thanks{The work of KA was supported by the project ETH-34 20-2 and the work of BD was supported by H2020 FETOpen project BOHEME under grant agreement No.~863179.}}
\author{Konstantinos Alexopoulos\thanks{Department of Mathematics, ETH Zurich, R\"amistrasse 101, CH-8092 Zurich, Switzerland, konstantinos.alexopoulos@sam.math.ethz.ch .} \and Bryn Davies\thanks{Department of Mathematics, Imperial College London, 180 Queen's Gate, London SW7 2AZ, United Kingdom, bryn.davies@imperial.ac.uk .}}
\date{}
\begin{document}

\maketitle

\begin{abstract}
    Halide perovskites are promising materials with many significant applications in photovoltaics and optoelectronics. In this paper, we use integral methods to quantify the resonant properties of halide perovskite nano-particles. We prove that, for arbitrarily small particles, the subwavelength resonant frequencies can be expressed in terms of the eigenvalues of the Newtonian potential associated with its shape. We also characterize the hybridized subwavelength resonant frequencies of a dimer of two halide perovskite particles. Finally, we examine the specific case of spherical resonators and demonstrate that our new results are consistent with previous works.
\end{abstract}

\tableofcontents

\section{Introduction}

Halide perovskites are set to be at the center of the next generation of electromagnetic devices \cite{JKM,MFTHBPZK,S}. They are composed of crystalline lattices which have octohedral shapes and contain atoms of heavier halides, such as chlorine, bromine and iodine \cite{AM}. Their excellent optical and electronic properties, combined with being cheap and easy to manufacture, have paved the way for a perovskite revolution. A particular benefit of halide perovskites is that their high absorption coefficient enables microscopic devices (measuring only a few hundred nanometres) to absorb the complete visible spectrum. Thus, we are able to design very small devices that are lightweight and compact while also being low cost and efficient. Research is ongoing to develop perovskites capable of fulfilling their theoretical capabilities for use in applications such as optical sensors \cite{GPLLZZSQKJ}, solar cells \cite{S} and light-emitting diodes \cite{WKYZZPLBHL}. 

The dielectric permittivity of a halide perovskite $\ve$ is given, in terms of the frequency $\w$ and wavenumber $k$ by \cite{MFTHBPZK}
$$
\ve(\w,k) = \ve_0 + \frac{\w_p^2}{\w^2_{exc} - \w^2 + \hslash \w_{exc} k^2 \m_{exc}^{-1} - i\g\w},
$$
where $\w_{exc}$ is the frequency of the excitonic transition, $\w_p$ is the strength of the dipole oscillator, $\g$ is the damping factor, $\m_{exc}$ is related to the non-local response and $\ve_0$ is the background dielectric constant. We refer to \cite{MFTHBPZK} for the values of these constants for different halide perovskites at room temperature. Meanwhile, the dispersion relation for the halide perovskites is observed \cite{HZDAMLLC} to take the simplified form
$$
n^2 = 1 + \frac{S_0 \l_0^2}{1 - \l_0^2 k^2},
$$
where $n$ is the refractive index, $k$ is the wavenumber and $S_0$ and $\l_0$ are positive constants that describe the average oscillator strength, see \cite{HZDAMLLC} for the values for some different halide perovsksites.

Dielectric nano-particles, and other electromagnetic metamaterials, have been studied using various techniques. In the case of extreme material parameters, such as small particle size or large material contrast, asymptotic methods can be used \cite{ADH,ADFMS,G}. Likewise,  homogenisation has been used to characterise materials with periodic micro-structures \cite{BBF, GZ}. Multiple scattering formulations are popular, particularly when a convenient choice of geometry (e.g. cylindrical or spherical resonators) facilitates explicit formulas \cite{TB}. In this paper, we will use integral methods to study a general class of geometries and will extend the previous theory, e.g. \cite{ADFMS}, to the case of dispersive materials.

For simplicity, we consider the Helmholtz equation as a model of the propagation of time-harmonic waves, and for the permittivity relation, we consider the form
\begin{align}\label{prmtvt}
    \ve(\w,k) = \ve_0 + \frac{\a}{\b - \w^2 + \h k^2 - i \g \w},
\end{align}
where $\a,\b,\g,\h$ are positive constants. We will use an approach based on representing the scattered solution using a Lippmann-Schwinger integral formulation and then using asymptotic methods to characterise the resonant frequencies in terms of the eigenvalues of the resulting integral operator (which turns out to be the Newtonian potential). This approach can handle a very general class of resonator shapes and can be adapted to study solutions to the Maxwell's equations \cite{ALZ}. A similar method was used in \cite{ADFMS} for a simpler, non-dispersive setting. This paper shows that the asymptotic theory developed in \cite{ADFMS} and elsewhere can be developed to model real-world settings and can be used to influence high-impact design problems.

In section~\ref{sec:single} of this paper we will introduce the problem setting and retrieve its integral formulation. Then, we will study the one particle case for three and two dimensions, using integral techniques to formulate the subwavelength resonant problem and asymptotic approximations to study the resonant frequencies. In section~\ref{sec:two}, we will use these methods to describe the hybridization of two halide perovskite resonators. Again, we treat the three- and two-dimensional cases separately. Passing from the integral to a matrix formulation of the problem, we obtain the hybridized subwavelength resonant frequencies. Finally, we examine the case of spherical resonators, making use of the fact that eigenvalues and eigenfunctions of the Newtonian potential can be computed explicitly in this case (see also \cite{KS}). We show that our findings are qualitatively consistent with the ones of \cite{ADFMS}. Hence, we show that the asymptotic techniques used in \cite{ADFMS} can be generalized to less straightforward and more impactful physical settings.

\section{Single Resonators} \label{sec:single}

\subsection{Problem Setting}

Consider a single resonator occupying a bounded domain $\Omega \subset \mathbb{R}^d$, for $d\in \{2,3\}$. We assume that the particle is non-magnetic, so that the magnetic permeability $\m_0$ is constant on all of $\mathbb{R}^d$. We will consider a time-harmonic wave with frequency $\w\in\mathbb{C}$ (which we assume to have positive real part). The wavenumber in the background $\mathbb{R}^d \setminus \overline{\W}$ is given by $k_0:=\w\ve_0\m_0$ and we will use $k$ to denote the wavenumber within $\W$. We, then, consider the following Helmholtz model for light propagation:
\begin{align}\label{Helmholtz problem}
    \begin{cases}
    \D u + \w^2 \ve(\w,k)\m_0 u = 0 \ \ \ &\text{ in } \W, \\
    \D u + k_0^2 u = 0 &\text{ in } \mathbb{R}^d\setminus\overline{\W}, \\
    u|_+ - u|_-=0 &\text{ on } \partial\W, \\
    \frac{\partial u}{\partial \n}|_+ - \frac{\partial u}{\partial \n}|_- = 0 &\text{ on } \partial\W, \\
    u(x) - u_{in}(x) &\text{ satisfies the outgoing radiation condition as } |x|\to \infty,\\
    \end{cases}
\end{align}
where $u_{in}$ is the incident wave, assumed to satisfy
\begin{align*}
    (\D + k_0^2)u_{in}=0 \ \ \text{ in } \mathbb{R}^d,
\end{align*}
and the appropriate outgoing radiation condition is the Sommerfeld radiation condition, which requires that
\begin{equation}\label{Sommerfeld}
    \lim_{|x|\to\infty} |x|^{\frac{d-1}{2}} \left( \frac{\partial}{\partial |x|} - i k_0 \right)(u(x) - u_{in}(x))=0.
\end{equation}
In particular, we are interested in the case of small resonators. Thus, we will assume that there exists some fixed domain $D$ such that $\W$ is given by
\begin{equation}
    \W = \d D + z,
\end{equation}
for some position $z\in\mathbb{R}^d$ and characteristic size $0<\d\ll1$. Then, making a change of variables, the Helmholtz problem (\ref{Helmholtz problem}) becomes
\begin{align}\label{2.4}
    \begin{cases}
    \D u + \d^2 \w^2 \ve(\w,k)\m_0 u = 0 \ \ \ &\text{ in } D, \\
    \D u + \d^2 k_0^2 u = 0 &\text{ in } \mathbb{R}^d\setminus\overline{D}, \\
    \end{cases}
\end{align}
along with the same transmission conditions on $\partial D$ and far-field behaviour. The behaviour of resonators which is of interest is when $\d \ll k_0^{-1}$, meaning the system can be described as being \emph{subwavelength}. We will study this by performing asymptotics in the regime that the frequency $\w$ is fixed while the size $\d\to0$.\\
We will characterise solutions to (\ref{Helmholtz problem}) in terms of the system's resonant frequencies. For a given wavenumber $k$, we define $\w=\w(k)$ to be a \emph{resonant frequency} if it is such that there exists a non-trivial solution $u$ to (\ref{Helmholtz problem}) in the case that $u_{in}=0$.

\subsection{Integral Formulation}

Let $G(x,k)$ be the outgoing Helmholtz Green's function in $\mathbb{R}^d$, defined as the unique solution to
\begin{align*}
    (\D + k^2) G(x,k) = \d_0(x) \ \ \text{ in } \mathbb{R}^d,
\end{align*}
along with the outgoing radiation condition (\ref{Sommerfeld}). It is well known that $G$ is given by
\begin{align}
    G(x,k)=
    \begin{cases}
    -\frac{i}{4}H^{(1)}_0(k|x|), \ \ \ &d=2,\\
    -\frac{e^{ik|x|}}{4\pi|x|}, &d=3,\\
    \end{cases}
\end{align}
where $H_0^{(1)}$ is the Hankel function of first kind and order zero.

\begin{theorem}[Lippmann-Schwinger Integral Representation Formula] \label{Integral representation}
The solution to the Helmholtz problem (\ref{Helmholtz problem}) is given by
\begin{align}\label{Lipmann-Schwinger}
    u(x)-u_{in}(x) = -\d^2 \w^2 \xi(\w,k) \int_D G(x-y,\d k_0) u(y) \upd y, \ \ x\in\mathbb{R}^d,
\end{align}
where the function $\xi:\mathbb{C}\to\mathbb{C}$ describes the permittivity contrast between $D$ and the background and is given by
$$
\xi(\w,k) = \m_0(\ve(\w,k) - \ve_0).
$$
\end{theorem}

\begin{proof}
We see from (\ref{2.4}) that $\D u + \d^2\w^2\ve(\w,k)\m_0u=0$ in $D$, so it holds that
\begin{align*}
    \D u +\d^2 k_0^2 u = -\d^2 \w^2 \xi(\w,k)u,
\end{align*}
where $\xi(\w,k) = \m_0(\ve(\w,k) - \ve_0)$. Therefore, the Helmholtz problem $(\ref{2.4})$ becomes
\begin{align*}
    (\D + \d^2 k_0^2)(u(y) - u_{in}(y)) = -\d^2 \w^2 \xi(\w,k)u(y) \chi_D(y), \ \ y\in\mathbb{R}^d \setminus\partial D,
\end{align*}
with $\chi_D$ being the indicator function of the set $D$. Then, we know that for $x\in\mathbb{R}^d$ and $y\in\mathbb{R}^d\setminus\partial D$, the following identity holds:
\begin{align*}
    \nabla_y \cdot \Big[  G(x-y,\d k_0) \nabla_y \Big( u(y)&-u_{in}(y) \Big) - \Big( u(y)-u_{in}(y) \Big) \nabla_y G(x-y,\d k_0) \Big] \\
    &= -\d^2 \w^2 \xi(\w,k)u(y)G(x-y,\d k_0) \chi_D(y) - \Big( u(y)-u_{in}(y) \Big)\d_0(x-y).
\end{align*}
Let $S_R$ be a large sphere with radius $R$ large enough so that $D \subset S_R$. Integrating the above identity over $y \in S_r \setminus \partial D$ and letting $R\to+\infty$, we can use the radiation condition (\ref{Sommerfeld}) in the far field and the transmission conditions on $\partial D$ to obtain the desired integral representation formula.
\end{proof}

We are interested in understanding how the formula from Theorem~\ref{Integral representation} behaves in the case that $\d$ is small. In particular, we wish to understand the operator $K^{\d k_0}_D: L^2(D)\to L^2(D)$ given by
\begin{align}
K^{\d k_0}_D[u](x) = - \int_{\W} G(x-y,\d k_0) u(y) \upd y, \ \ \ \ x \in D.    
\end{align}
The Helmholtz Green’s function has helpful asymptotic expansions which facilitate this. However, the behaviour is quite different in two and three dimensions, so we must now consider these two settings separately. We will first work on the three-dimensional case, and then treat the two-dimensional setting as the asymptotic expansions are more complicated.

\subsection{Reformulation as a Subwavelength Resonance Problem}

In order to reveal the behaviour of the system of nano-particles, we will characterise the properties of the operator $K^{\d k_0}_D$. We observe that the Lippmann-Schwinger formulation (\ref{Lipmann-Schwinger}) of the problem is equivalent to
\begin{align*}
    (u-u_{in})(x) = \d^2 \w^2 \xi(\w,k) K^{\d k_0}_D[u](x).
\end{align*}
This is equivalent to
\begin{align*}
    (I - \d^2 \w^2 \xi(\w,k) K^{\d k_0}_D)[u](x) = u_{in}(x), 
\end{align*}
which gives
\begin{align*}
    u(x) = (I - \d^2 \w^2 \xi(\w,k) K^{\d k_0}_D)^{-1}[u_{in}](x), 
\end{align*}
for all $x\in D$, where $I$ denotes the identity operator. Then, the subwavelength resonance problem is to find $\w \in \mathbb{C}$ close to 0, such that the operator $(I - \d^2 \w^2 \xi(\w,k) K^{\d k_0}_D)^{-1}$ is singular, or equivalently, such that there exists $u \in L^2(D)$, $u \ne 0$ with
\begin{align}\label{pb}
    u(x) - \d^2 \w^2 \xi(\w,k) \int_D G(x-y,\d k_0) u(y) \upd y = 0, \ \ \ \text{ for } x \in D.
\end{align}

\subsection{Three Dimensions}

We consider the three-dimensional case, $d=3$. Let us define the Newtonian potential on $D$ to be $K^{(0)}_D: L^2(D)\to L^2(D)$ such that
\begin{equation}
    K^{(0)}_D[u](x) := - \int_D u(y) G(x-y,0) \upd y = - \frac{1}{4\pi} \int_D \frac{1}{|x-y|}u(y)\upd y.
\end{equation}
Similarly, we define the operators $K^{(n)}_D: L^2(D) \to L^2(D)$, for $n=1,2,\dots,$ as
\begin{equation}\label{Kn}
    K^{(n)}_D[u](x) = - \frac{i}{4\pi} \int_D \frac{(i|x-y|)^{n-1}}{n!}u(y)\upd y.
\end{equation}
Then, in order to capture the behaviour of $K^{\d k_0}_D$ for small $\d$, we can use an asymptotic expansion in terms of small $\delta$.
\begin{lemma} \label{Taylor d=3}
Suppose that $d=3$. The operator $K^{\d k_0}_D$ can be rewritten as
\begin{align*}
    K^{\d k_0}_D = \sum_{n=0}^{\infty} (\d k_0)^n K^{(n)}_D,
\end{align*}
where the series converges in the $L^2(D) \to L^2(D)$ operator norm if $\d k_0$ is small enough.
\end{lemma}
\begin{proof}
This follows from an application of the Taylor expansion on the operator $K^{\d k_0}_D$. Indeed, using Taylor expansion on the second variable of $G$, we can see that
$$
G(x-y,\d k_0) = \sum_{n=0}^{\infty} (\d k_0)^n \frac{\partial^n G}{\partial k^n}(x-y,k)\Big|_{k=0}.
$$
Since we are working in three dimensions, 
$$
G(x,k) = - \frac{e^{ik|x|}}{4\pi|x|}
$$
and, thus,
$$
G(x-y,0) = -\frac{1}{4\pi|x-y|} \quad\text{and}\quad \frac{\partial^n G}{\partial k^n}(x-y,k) = -\frac{i}{4\pi} \frac{(i|x-y|)^{n-1}}{n!} \quad\text{for }n>0.
$$
Hence,
%
\begin{align*}
    G(x-y,\d k_0) &= \sum_{n=0}^{\infty} (\d k_0)^n \frac{\partial^n G}{\partial k^n}(x-y,k)\Big|_{k=0} = \sum_{n=0}^{\infty} (\d k_0)^n \left(-\frac{i}{4\pi}\right) \frac{(i|x-y|)^{n-1}}{n!},
\end{align*}
and then, multiplying by $u$ and integrating over $D$, gives
\begin{align*}
    K^{\d k_0}_D[u](y) &= \sum_{n=0}^{\infty} (\d k_0)^n K^{(n)}_D[u](y),
\end{align*}
which is the desired result.
\end{proof}

\subsubsection{Eigenvalue Calculation}

In order to study the operator $K^{\d k_0}_D$, we need to find its eigenvalues. From Lemma~\ref{Taylor d=3}, if $k_0$ is fixed then we can write our operator as
\begin{align} \label{expansion}
K^{\d k_0}_D = K^{(0)}_D + \d k_0 K^{(1)}_D + O(\d^2) \ \ \ \text{ as } \d\to0.
\end{align}
Since we know that $K^{(0)}_D$ is self-adjoint, we know that it admits eigenvalues. Let us denote such an eigenvalue by $\l_0$, and by $u_0$ the associated eigenvector. Let us now consider the problem
\begin{equation}
K^{\d k_0}_D u_{\d} = \l_{\d} u_{\d},
\end{equation}
where $\l_{\d}$ denotes an eigenvalue of the operator $K^{\d k_0}_D$ and $u_{\d}$ denotes the associated eigenvector. Using this eigenvalue problem, we will find the resonant frequency $\w_s$ and the associated wavenumber $k_s$ of the halide perovskite nano-particle.

\begin{proposition}
Let $\l_{\d}$ denote an eigenvalue of the operator $K^{\d k_0}_D$ in dimension three. Then, if $\delta$ is small, it is approximately given by
\begin{align}\label{l_e approx d=3}
    \l_{\d} \approx \l_0 + \d k_0 \langle K^{(1)}_D u_0, u_0 \rangle.
\end{align}
\end{proposition}
\begin{proof}
We start by truncating the $O(\d^2)$ term in the expansion (\ref{expansion}). Then, we have that
\begin{align*}
    K^{\d k_0}_D u_{\d} = \l_{\d} u_{\d} \ &\Leftrightarrow \ (K^{(0)}_D + \d k_0 K^{(1)}_D)u_{\d} = \l_{\d} u_{\d} \ \Leftrightarrow \ \langle (K^{(0)}_D + \d k_0 K^{(1)}_D)u_{\d} , u_0 \rangle = \langle \l_{\d} u_{\d}, u_0 \rangle.
\end{align*}
Since $K^{(0)}_D$ is self-adjoint, we have that
\begin{align*}
     \ \l_0\langle u_{\d} , u_0 \rangle + \d k_0 \langle K^{(1)}_D u_{\d}, u_0 \rangle = \l_{\d} \langle u_{\d},u_0 \rangle.
\end{align*}
Finally, since $u_{\d}\approx u_0$ we find that
\begin{align*}
    \l_{\d} = \l_0 + \d k_0 \frac{\langle K^{(1)}_D u_{\d}, u_0 \rangle}{\langle u_{\d},u_0 \rangle}
    \approx \l_0 + \d k_0 \langle K^{(1)}_D u_0, u_0 \rangle,
\end{align*}
which is the desired result.
\end{proof}
The following corollary is a direct consequence of Proposition~\ref{l_e approx d=3}:
\begin{corollary}
Let $\l_{\d}$ denote an eigenvalue of the operator $K^{\d k_0}_D$ in three dimensions. Then, if $\delta$ is small, it is approximately given by
\begin{align} \label{l_e approx d=3 vol.2}
    \l_{\d} \approx \l_0 - \frac{i}{4\pi} \d k_0 \mathbb{B},
\end{align}
where $\mathbb{B}:=(\int_D u_0(y) \upd y)^2$ is a constant.
\end{corollary}
\begin{proof}
From (\ref{Kn}), we get that
\begin{align*}
    K^{(1)}_D[u](x) = -\frac{i}{4\pi} \int_{D} \frac{ (i|x-y|)^{1-1} }{1!}u(y)\upd y = -\frac{i}{4\pi} \int_{D} u(y) \upd y.
\end{align*}
Let us define the constant $\mathbb{B}:=(\int_D u_0(y) \upd y)^2$. Then, we observe that
\begin{align*}
    \langle K^{(1)}_D u_0, u_0 \rangle = \int_D \left(-\frac{i}{4\pi}\right)u_0(y) \int_D \overline{u_0}(x)\upd x \upd y = -\frac{i}{4\pi} \left( \int_D u_0(y) \upd y  \right)^2 = -\frac{i}{4\pi}\mathbb{B}.
\end{align*}
Substituting into (\ref{l_e approx d=3}) gives the desired result.
\end{proof}

\subsubsection{Frequency and Wavenumber}

Let us now find the resonant frequency $\w_{\d}$ and wavenumber $k_{\d}$ associated to this eigenvalue, which will also constitute the basis of our analysis of the operator $K^{\d k_0}_D$. From (\ref{pb}), we see that if $u=u_\ve$, then $1 = \d^2\w^2 \xi(\w,k) \l_{\d}$ so, using Corollary~\ref{l_e approx d=3 vol.2} we have that
\begin{align} \label{eq:epsilon1}
    \ve(\w,k) = \frac{1}{\m_0 \d^2 \w^2 \left( \l_0 - \frac{i}{4\pi} \d k_0 \mathbb{B} \right)} + \ve_{0}.
\end{align}
In order to study halide perovskite particles, we want the permittivity $\ve(\w,k)$ to be given by \eqref{prmtvt}, that is,
\begin{equation} \label{eq:epsilon2}
    \ve(\w,k) = \ve_0+\frac{\a}{\b-\w^2+\h k^2 - i\g\w},
\end{equation}
where $\a,\b,\g,\h$ are positive constants. Comparing the two expressions \eqref{eq:epsilon1} and \eqref{eq:epsilon2} we see that
\begin{align}
    \begin{cases}
    \a \m_0 \d^2 \w^2 \l_0 - \b + \w^2 - \h k^2 = 0, \\
    \g\w - \m_0 \frac{1}{4\pi}\a\d^3\w^2k_0\mathbb{B}=0.
    \end{cases} \label{system d=3}
\end{align}
We study these two equations separately. First, we look at the second equation of (\ref{system d=3}). We have that
\begin{align*}
    \w\left(\g - \m_0 \frac{1}{4\pi}\a\d^3\w k_0\mathbb{B}\right)=0,
\end{align*}
meaning that
\begin{align*}
    \w = 0 \quad \text{ or } \quad \w = \frac{4\pi\g}{\a\m_0\d^3 k_0 \mathbb{B}}.
\end{align*}
For $\w = \frac{4\pi\g}{\a\m_0\d^3 k_0 \mathbb{B}}$, we obtain that
\begin{align*}
    \h k^2 = (1 + \a \m_0 \d^2 \l_0)\w^2 - \b,
\end{align*}
    which has solutions
\begin{align}
    k = \pm \sqrt{ \frac{16\pi^2\g^2(1 + \a \m_0 \d^2 \l_0)}{\a^2\m_0^2\d^6 k_0^2 \mathbb{B}^2 \h} - \frac{\b}{\h}}.
\end{align}
The case of $\w=0$ is not of physical interest here. Thus, denoting this specific frequency by $\w_{\d}$ and the associated wavenumber by $k_{\d}$, we will work with
\begin{align}\label{w_e k_e for d=3}
    \w_{\d} = \frac{4\pi\gamma}{\a\m_0\d^3k_0\mathbb{B}} \quad \text{ and } \quad k_{\d} =  \sqrt{ \frac{16\pi^2\g^2(1 + \a \m_0 \d^2 \l_0)}{\a^2\m_0^2\d^6 k_0^2 \mathbb{B}^2 \h} - \frac{\b}{\h}},
\end{align}
where we have chosen the wavenumber $k_{\d}$ to be positive.

\subsubsection{Asymptotic Analysis} \label{sec:AA_3}

Let us now return to the problem of studying the singularities of the operator $(I - \d^2 \w^2 \xi(\w,k) K^{\d k_0}_D)^{-1}$. We have the following equivalence:
\begin{align*}
    (I - \d^2 \w^2 \xi(\w,k) K^{\d k_0}_D)^{-1} = 0 \ \Leftrightarrow \ \left( I - \d^2 \w^2 \xi(\w,k) \sum_{n=0}^{\infty}(\d k_0)^n K^{(n)}_D \right)^{-1}=0.
\end{align*}
We define
$$
A_n := \d^2 \w^2 \xi(\w,k) K^{(n)}_D = - \d^2 \w^2 \xi(\w,k) \frac{i}{4\pi} \int_D \frac{(i|x-y|)^{n-1}}{n!}u(y)\upd y.
$$
Then, it holds that
\begin{align*}
    \left( I - \d^2 \w^2 \xi(\w,k) \sum_{n=0}^{\infty}(\d k_0)^n K^{(n)}_D \right)^{-1} &=  \left( I - \sum_{n=0}^{\infty}(\d k_0)^n A_n \right)^{-1} \\
    &= \sum_{i=0}^{\infty} \left( (I-A_0-\d k_0 A_1)^{-1} \sum_{n=2}^{\infty} (\d k_0)^n A_n \right)^i(I - A_0 -\d k_0 A_1)^{-1}\\
    &= (I - A_0 -\d k_0 A_1)^{-1}\\
    &\ \ \  +(I - A_0 -\d k_0 A_1)^{-1} (\d k_0)^2 A_2 (I - A_0 -\d k_0 A_1)^{-1} + O(\d^4).
\end{align*}
Thus, the above equivalence yields
\begin{align*}
    (I - \d^2 \w^2 \xi(\w,k) K^{\d k_0}_D)^{-1} = 0 \ \Leftrightarrow \ \left( I - \d^2 \w^2 \xi(\w,k) \sum_{n=0}^{\infty}(\d k_0)^n K^{(n)}_D \right)^{-1}=0 \ \Leftrightarrow\\
    (I - A_0 -\d k_0 A_1)^{-1}+(I - A_0 -\d k_0 A_1)^{-1} (\d k_0)^2 A_2 (I - A_0 -\d k_0 A_1)^{-1} + O(\d^4)=0.
\end{align*}
Using this expression, we obtain the following proposition.
\begin{proposition} \label{prop:24}
Let $d=3$ and let $\w_{\d}$ be defined by (\ref{w_e k_e for d=3}). Then, as $\delta\to0$, the $O(\d^4)$ approximation of the subwavelength resonant frequencies $\w_s$ and the associated wavenumbers $k_s$ of the single halide perovskite resonator $\Omega = \d D + z$ satisfy
\begin{equation}\label{Rose in Harlem}
    1 - \d^2 \w_s^2 \xi(\w_s,k_s) \l_{\d} = -\d^4 k_0^2 \w_s^2 \xi(\w_s,k_s) \langle K^{(2)}_D[u_{\d}], u_{\d} \rangle.
\end{equation}
\end{proposition}
\begin{proof}
For $\psi \in L^{2}(D)$, and dropping the $O(\d^4)$ term, we have
\begin{equation}\label{expansion of s.r.p.}
    \Big[(I - A_0 -\d k_0 A_1)^{-1}+(I - A_0 -\d k_0 A_1)^{-1} (\d k_0)^2 A_2 (I - A_0 -\d k_0 A_1)^{-1}\Big][\psi]=0.
\end{equation}
We apply a pole-pencil decomposition, as it is defined in \cite{AFKRYZ}, to the term $(I - A_0 -\d k_0 A_1)^{-1}[\psi]$ and obtain
$$
(I - A_0 -\d k_0 A_1)^{-1}[\psi] = \frac{\langle u_{\d},\psi \rangle u_{\d}}{1 - \d^2 \w_s^2 \xi(\w_s,k_s) \l_{\d}} + R(\d)[\psi],
$$
where the remainder term $R(\d)[\psi]$ can be dropped. Hence, \eqref{expansion of s.r.p.} is, at leading order, equivalent to
\begin{align*}
    \frac{\langle u_{\d},\psi \rangle u_{\d}}{1 - \d^2 \w^2 \xi(\w_s,k_s) \l_{\d}} + \frac{\langle u_{\d},\psi \rangle u_{\d}}{1 - \d^2 \w_s^2 \xi(\w_s,k_s) \l_{\d}} (\d k_0)^2 A_2 \frac{\langle u_{\d},\psi \rangle u_{\d}}{1 - \d^2 \w_s^2 \xi(\w_s,k_s) \l_{\d}} = 0,
\end{align*}
which reduces to
\begin{align*}
    \frac{u_{\d}}{1 - \d^2 \w_s^2 \xi(\w_s,k_s) \l_{\d}} + \frac{u_{\d}\langle K^{(2)}_D[u_{\d}], u_{\d} \rangle}{(1 - \d^2 \w_s^2 \xi(\w_s,k_s) \l_{\d})^2} \d^4 k_0^2 \w_s^2 \xi(\w_s,k_s) = 0,
\end{align*}
which can be rearranged to give the desired result.
\end{proof}
\begin{remark}
In the appendix, we will recover a formula for $\langle K^{(2)}_D[u_{\d}], u_{\d} \rangle$.
\end{remark}
Before we move on to the consequences of this proposition, we recall that
\begin{align*}
    K^{(2)}_D[u](x) = -\frac{i}{4\pi} \int_D \frac{i|x-y|}{2}u(y)\upd y = \frac{1}{8\pi} \int_D |x-y|u(y)\upd y.
\end{align*}
Thus, we can define the constant $\mathbb{F}$ to be such that
\begin{align*}
    \langle K^{(2)}_D[u_{\d}], u_{\d} \rangle &= \int K^{(2)}_D[u_{\d}](x) \overline{u_{\d}}(x) \upd x
    = \frac{1}{8\pi} \int_D \int_D |x-y| u_{\d}(y) \overline{u_{\d}}(x) \upd y \upd x
    =: \frac{1}{8\pi} \mathbb{F}.
\end{align*}
So, from (\ref{Rose in Harlem}), we see that
\begin{align}\label{Rose in Harlem 2}
    1 - \d^2 \w_s^2 \xi(\w_s,k_s) \l_{\d} = - \frac{\d^4 k_0^2 \w_s^2 \xi(\w_s,k_s)}{8\pi} \mathbb{F}.
\end{align}
Then the following two corollaries are immediate consequences of the Proposition~\ref{prop:24}.
\begin{corollary}
Let $d=3$. Then, as $\delta\to0$, the $O(\d^8)$ approximation of the subwavelength resonant frequencies of the halide perovskite resonator $\Omega = \d D + z$ are given by
\begin{equation}
    1 - \frac{\w_s^2}{\w_{\d}^2} \l_{\d} (\ve(\w_s,k_s) - \ve_0) \frac{64 \pi^2 \g^2}{\a^2 \m_0 \d^4 k_0^2 \mathbb{B}^2} = - \frac{\d^4 \w_s^2 k_0^2 \xi(\w_s,k_s)}{8\pi} \mathbb{F}.
\end{equation}
\end{corollary}
\begin{proof}
By a direct calculation and using (\ref{w_e k_e for d=3}), we have
\begin{align*}
    1 - \d^2 \w_s^2 \xi(\w_s,k_s) \l_{\d} &\ = 1 - \d^2 \w_s^2 \m_0(\ve(\w_s,k_s)-\ve_0) \l_{\d} \\
    &\overset{(\ref{w_e k_e for d=3})}{=} 1 - \d^2 \frac{\w_s^2}{\w_{\d}^2} \m_0(\ve(\w_s,k_s)-\ve_0) \l_{\d} \frac{64 \pi^2 \g^2}{\a^2 \m_0^2 \d^4 k_0^2 \mathbb{B}^2}\\
    &\ = 1 - \frac{\w_s^2}{\w_{\d}^2} \l_{\d} (\ve(\w_s,k_s) - \ve_0) \frac{64 \pi^2 \g^2}{\a^2 \m_0 \d^4 k_0^2 \mathbb{B}^2}.
\end{align*}
Then, using (\ref{Rose in Harlem 2}), the result follows.
\end{proof}
\begin{corollary}
Let $d=3$. Then, as $\delta\to0$, the $O(\d^4)$ approximation of the subwavelength resonant frequencies of the halide perovskite resonator $\Omega = \d D + z$ can be computed as
\begin{equation}
    1 - \d^2 \w_s^2 \xi(\w_s,k_s) \l_{\d} = -\frac{\w_s^2}{\w_{\d}^2}(\ve(\w_s,k_s) - \ve_0) \frac{8\pi\g^2\mathbb{F}}{\a^2 \m_0 \d^2 \mathbb{B}^2}.
\end{equation}
\end{corollary}
\begin{proof}
Again, we can calculate this directly:
\begin{align*}
    \frac{\d^4 k_0^2 \w_s^2 \xi(\w_s,k_s)}{8\pi} \mathbb{F} &\ = \frac{1}{8\pi} \d^4 \frac{\w_s^2}{\w_{\d}^2} k_0^2 \m_0 (\ve(\w_s,k_s) - \ve_0) \w_{\d}^2 \\
    &\overset{(\ref{w_e k_e for d=3})}{=} \frac{1}{8\pi} \frac{\w_s^2}{\w_{\d}^2} \d^4 k_0^2 \m_0 (\ve(\w_s,k_s) - \ve_0) \frac{64 \pi^2 \g^2}{\a^2 \m_0^2 \d^4 k_0^2 \mathbb{B}^2}\\
    &= \frac{\w_s^2}{\w_{\d}^2} (\ve(\w_s,k_s) - \ve_0) \frac{8\pi\g^2}{\a^2 \m_0 \d^2 \mathbb{B}^2}.
\end{align*}
Then, using (\ref{Rose in Harlem 2}), the result follows.
\end{proof}
We finish our analysis of the three-dimensional case with the following proposition.
\begin{proposition}
Let $d=3$. For $\w$ close to the resonant frequency $\w_s$, the field scattered by the halide perovskite nano-particle $\Omega = \d D + z$ can be approximated by
\begin{align*}
    u(x)-u_{in}(x) \approx \frac{1 + \d^2 \w^2 G(x-y,\d k_0) \xi(\w,k)}{\d^2 \w^2 \left( \l_{\d} - \l_0 + \frac{i}{4\pi}\d k_0 \mathbb{B} \right)\xi(\w,k)} \langle u_{in}, u_{\d} \rangle \int_{D} u_{\d}.
\end{align*}
\end{proposition}
\begin{proof}
Our goal is to find $u\in L^{2}(D)$, $u\ne 0$ such that (\ref{pb}) is satisfied. Using the pole-pencil decomposition, we can rewrite
\begin{align*}
    u(x)-u_{in}(x) \approx -\d^2 \w^2 \xi(\w,k) &G(x-y,\d k_0) \frac{\langle u_{in}, u_{\d} \rangle \int_{D} u_{\d}}{1-\d^2\w^2\xi(\w,k)\l_{\d}} \\
    &+ \d^{4}k_0^2\w^2\xi(\w,k) \frac{\langle K^{(2)}_D[u_{\d}], u_{\d} \rangle \langle u_{in}, u_{\d} \rangle \int_{D} u_{\d}}{(1-\d^2\w^2\xi(\w,k)\l_{\d})^2}.
\end{align*}
Using (\ref{Rose in Harlem}) and plugging it in the above expression, we obtain
\begin{align*}
    u(x)-u_{in}(x) &\approx \ -\d^2 \w^2 \xi(\w,k)G(x-y,\d k_0) \frac{\langle u_{in}, u_{\d} \rangle \int_{D} u_{\d}}{-\d^4 k_0^2 \w^2 \xi(\w,k) \langle K^{(2)}_D[u_{\d}], \u_{\d} \rangle} \\
    &\ \ \ \ \ \ \ + \d^{4}k_0^2\w^2\xi(\w,k) \frac{\langle K^{(2)}_D[u_{\d}], u_{\d} \rangle \langle u_{in}, u_{\d} \rangle \int_{D} u_{\d}}{\d^8 k_0^4 \w^4 \xi(\w,k)^2 \langle K^{(2)}_D[u_{\d}], \u_{\d} \rangle^2} \\
    &= \ \frac{G(x-y,\d k_0)\langle u_{in}, u_{\d} \rangle \int_{D} u_{\d}}{\d^2 k_0^2 \langle K^{(2)}_D[u_{\d}], u_{\d} \rangle} + \frac{\langle u_{in}, u_{\d} \rangle \int_{D} u_{\d}}{\d^4 k_0^2 \w^2 \xi(\w,k) \langle K^{(2)}_D[u_{\d}], u_{\d} \rangle}. 
\end{align*}
Thus, defining the constant $\mathbb{F}:=8\pi \langle K^{(2)}_D[u_{\d}], u_{\d} \rangle$, we obtain
\begin{align}\label{damso}
     u(x)-u_{in}(x) \approx \ \frac{8\pi\Big(1+\d^2\w^2G(x-y,\d k_0)\xi(\w,k)\Big)}{\d^4 k_0^2 \w^2 \xi(\w,k) \mathbb{F}} \langle u_{in}, u_{\d} \rangle \int_{D} u_{\d}.
\end{align}
From the Appendix~\ref{app1}, we have that
$$
\mathbb{F} = \frac{8\pi}{\d^2 k_0^2} \left( \l_{\d} - \l_0 + \frac{i}{4\pi} \d k_0 \mathbb{B} \right).
$$
Plugging it into (\ref{damso}), we obtain
$$
 u(x)-u_{in}(x) \approx \frac{8\pi\Big(1+\d^2\w^2G(x-y,\d k_0)\xi(\w,k)\Big)}{\d^4 k_0^2 \w^2 \xi(\w,k) \mathbb{F}} \frac{\d^2 k_0^2}{8\pi} \frac{1}{\l_{\d} - \l_0 + \frac{i}{4\pi} \d k_0 \mathbb{B}} \langle u_{in}, u_{\d} \rangle \int_{D} u_{\d},
$$
from which the result follows.
\end{proof}

\subsection{Two Dimensions}

We now turn our attention to the two-dimensional setting. We define the Newtonian potential on $D$, $K^{(0)}_D: L^2(D)\to L^2(D)$, by
\begin{align*}
    K^{(0)}_D[u](x) := -\int_D u(y) G(x-y,0)\upd y = - \frac{1}{2\pi} \int_D \log|x-y| u(y) \upd y.
\end{align*}
We also define the operators $K^{(-1)}_D: L^2(D)\to L^2(D)$ and $K^{(1)}_D: L^2(D)\to L^2(D)$ by
\begin{align*}
    K^{(-1)}_D[u](x) := - \frac{1}{2\pi} \int_D u(y) \upd y \quad \text{ and } \quad K^{(n)}_D[u](x) := \int_D \frac{\partial^n}{\partial k^n}G(x-y,k)\Big|_{k=0} u(y) \upd y.
\end{align*}
Then, from the asymptotic expansion of the Hankel function, we have the following result.
\begin{lemma}
Suppose that $d=2$. Then, for fixed $k_0 \in \mathbb{C}$, the operator $K^{\d k_0}_D$ satisfies
\begin{align}\label{expansion d=2}
    K^{\d k_0}_D = \log(\d k_0 \hat{\g}) K^{(-1)}_D + K^{(0)}_D + (\d k_0)^2 \log( \hat{\g} \d k_0 ) K^{(1)}_D + O( \d^4 \log\d),
\end{align}
as $\delta\to0$, with convergence in the $L^2(D)\to L^2(D)$ operator norm and the constant $\hat{\g}$ being given by $\hat{\g}:=\frac{1}{2}k_0\exp(\g-\frac{i\pi}{2})$, where $\g$ is the Euler constant.
\end{lemma}

\subsubsection{Eigenvalue Calculation}

We use the same notation for $u_{\d}$ as in section 2.4.1. Let us consider the eigenvalue problem 
\begin{equation}
K^{\d k_0}_D u_{\d} = \l_{\d} u_{\d},
\end{equation}
where $\l_{\d}$ denotes an eigenvalue of the operator $K^{\d k_0}_D$, and $u_{\d}$ denotes an associated eigenvector. 
\begin{proposition}
Let $\l_{\d}$ denote an eigenvalue of the operator $K^{\d k_0}_D$ in dimension 2. Then, for small $\delta$, it is approximately given by:
\begin{align}\label{l_e approx d=2}
    \l_{\d} \approx \log(\d k_0 \hat{\g})\l_{-1} + \langle K^{(0)}_D u_{-1}, u_{-1} \rangle + (\d k_0)^2 \log(\d k_0 \hat{\g}) \langle K^{(1)}_D u_{-1}, u_{-1} \rangle,
\end{align}
where $\l_{-1}$ and $u_{-1}$ are an eigenvalue and the associated eigenvector of the potential $K^{(-1)}_D$.
\end{proposition}
\begin{proof}
We start by dropping the $O(\d^4\log(\d))$ term on the expansion (\ref{expansion d=2}). Then, we have that
\begin{align*}
    K^{\d k_0}_D u_{\d} = \l_{\d} u_{\d} \ &\Leftrightarrow \ (\log(\d k_0 \hat{\g}) K^{(-1)}_D + K^{(0)}_D + (\d k_0)^2 \log(\d k_0 \hat{\g}) K^{(1)}_D)u_{\d} = \l_{\d} u_{\d} \\
    &\Leftrightarrow \ \log(\d k_0 \hat{\g}) \l_{-1}\langle u_{\d} , u_0 \rangle + \langle K^{(0)}_D u_{\d}, u_{-1} \rangle + (\d k_0)^2 \log(\d k_0 \hat{\g}) \langle K^{(1)}_D u_{\d}, u_{-1} \rangle = \l_{\d} \langle u_{\d},u_{-1} \rangle.
\end{align*}
Therefore, assuming that $u_{\d} \approx u_{-1}$, we can see that 
\begin{equation*}
    \l_{\d} \approx \log(\d k_0 \hat{\g})\l_{-1} + \langle K^{(0)}_D u_{-1}, u_{-1} \rangle + (\d k_0)^2 \log(\d k_0 \hat{\g}) \langle K^{(1)}_D u_{-1}, u_{-1} \rangle,
\end{equation*}
which is the desired result.
\end{proof}
The following corollary is a direct consequence of the above proposition.
\begin{corollary}
Let $\l_{\d}$ denote an eigenvalue of the operator $K^{\d k_0}_D$ in dimension 2. Then, for small $\delta$, it is approximately given by
\begin{align} \label{l_e approx d=2 vol.2}
    \l_{\d} \approx \log(\d k_0 \hat{\g})\l_{-1} - \frac{\mathbb{P}}{2\pi} - \frac{i(\d k_0)^2 \log(\d k_0 \hat{\g}) \mathbb{G}}{4\pi},
\end{align}
where $\mathbb{P}$ and $\mathbb{G}$ are constants that depend on $u_{-1}$.
\end{corollary}
\begin{proof}
We observe that
\begin{align*}
    \langle K^{(0)}_D u_{-1}, u_{-1} \rangle &= \int_D \left( -\frac{1}{2\pi} \int_D \log|x-y| u_{-1}(y) \upd y \right) \overline{u_{-1}}(x) \upd x \\
    &= -\frac{1}{2\pi} \int_D \int_D \log|x-y| u_{-1}(y) \overline{u_{-1}}(x) \upd y \upd x
    =: - \frac{1}{2\pi} \mathbb{P}.
\end{align*}
Then, for $u\in L^2(D)$,
\begin{align*}
    K^{(1)}_D[u](x) &= \int_D \frac{\partial}{\partial k} G(x-y,k)\Big|_{k=0} u(y) \upd y = \int_D \frac{\partial}{\partial k} \left( -\frac{i}{4} H_0^{(1)}(k|x-y|) \right)\Big|_{k=0} u(y) \upd y= -\frac{i}{4\pi} \int_D \frac{u(y)}{|x-y|} \upd y,
\end{align*}
and so, we have
\begin{align*}
     \langle K^{(1)}_D u_{-1}, u_{-1} \rangle &= \int_D \left( -\frac{i}{4\pi} \int_D \frac{u_{-1}(y)}{|x-y|} \upd y \right) \overline{u_{-1}}(x) \upd x = -\frac{i}{4\pi} \int_D \int_D \frac{u_{-1}(y) \overline{u_{-1}}(x)}{|x-y|} \upd y \upd x =: -\frac{i}{4\pi} \mathbb{G}.
\end{align*}
Hence, from (\ref{l_e approx d=2}), we obtain the desired result.
\end{proof}

\subsubsection{Frequency and Wavenumber}

Let us now find the frequency $\w_{\d}$ and the wavenumber $k_{\d}$ associated to this eigenvalue, which will also constitute the basis of our analysis of the operator $K^{\d k_0}_D$. We see that (\ref{pb}) gives us that
\begin{align*}
    1 = \d^2\w^2 \xi(\w,k) \l_{\d} \ &\Leftrightarrow \ 1 =  \d^2\w^2 \m_0 (\ve(\w,k) - \ve_0) \l_{\d}.
\end{align*}
Using the expression (\ref{l_e approx d=2}) for $\l_{\d}$, we see that
\begin{align*}
    \ve(\w,k) = \frac{1}{\m_0 \d^2 \w^2 \left(  \log(\d k_0 \hat{\g})\l_{-1} - \frac{\mathbb{P}}{2\pi} - \frac{i(\d k_0)^2 \log(\d k_0 \hat{\g}) \mathbb{G}}{4\pi} \right)} + \ve_{0}.
\end{align*}
Arguing in the same way as in section 2.4.4 and comparing the two permittivity expressions, we obtain the following system:
\begin{align}\label{systemo}
    \begin{cases}
    4\pi\a\d^2\w^2\m_0\log(\d k_0 \hat{\g})\l_{-1} - 2\mathbb{P} \a\d^2\w^2\m_0 - 4\pi\b + 4\pi\w^2-4\pi\h k^2 = 0,\\
    -\a\d^4\w^2\m_0 k_0^2 \log(\d k_0 \hat{\g}) \mathbb{G} + 4\pi\g\w = 0.
    \end{cases}
\end{align}
From the second equation in (\ref{systemo}), we see that
\begin{align*}
     \w( -\a\d^4\w\m_0 k_0^2 \log(\d k_0 \hat{\g}) \mathbb{G} + 4\pi\g) = 0,
\end{align*}
which shows us that
\begin{align*}
    \w=0 \ \ \ \text{ or } \ \ \ \w = \frac{4\pi\g}{\a\d^4\m_0 k_0^2 \log(\d k_0 \hat{\g}) \mathbb{G}}.
\end{align*}
For $\w = \frac{4\pi\g}{\a\d^4\m_0 k_0^2 \log(\d k_0 \hat{\g}) \mathbb{G}}$, we obtain the equation
\begin{align*}
     4\pi\a\d^2\m_0\l_{-1}\log(\d k_0 \hat{\g})\frac{16\pi^2\g^2}{\a^2\d^8\m_0^2 k_0^4 \log(\d k_0 \hat{\g})^2 \mathbb{G}^2} &- 2\mathbb{P} \a\d^2\m_0 \frac{16\pi^2\g^2}{\a^2\d^8\m_0^2 k_0^4 \log(\d k_0 \hat{\g})^2 \mathbb{G}^2}\\
     &\quad- 4\pi\b + 4\pi \frac{16\pi^2\g^2}{\a^2\d^8\m_0^2 k_0^4 \log(\d k_0 \hat{\g})^2 \mathbb{G}^2} - 4\pi\h k^2=0,
\end{align*}
which has solutions
\begin{align*}
     k = \pm \sqrt{-\frac{\b}{\h} + \Big( 2\pi\a\d^2\m_0\l_{-1} \log(\d k_0 \hat{\g}) - \a\d^2\m_0\mathbb{P}+2\pi \Big)\frac{8\pi\g^2}{\h \a^2\d^8\m_0^2 k_0^4 \log(\d k_0 \hat{\g})^2 \mathbb{G}^2}}.
\end{align*}
Yet again, we discard the case of $\w=0$, as there is no physical interest. Denoting the frequency by $\w_{\d}$ and the wavenumber by $\l_{\d}$, we will work with
\begin{align}\label{w_e k_e for d=2}
\begin{split}
    &\w_{\d} = \frac{4\pi\g}{\a\d^4\m_0 k_0^2 \log(\d k_0 \hat{\g}) \mathbb{G}}, \\
    &k_{\d} = \sqrt{-\frac{\b}{\h} + \Big( 2\pi\a\d^2\m_0\l_{-1} \log(\d k_0 \hat{\g}) - \a\d^2\m_0\mathbb{P}+2\pi \Big)\frac{8\pi\g^2}{\h \a^2\d^8\m_0^2 k_0^4 \log(\d k_0 \hat{\g})^2 \mathbb{G}^2}},
\end{split}
\end{align}
where we have chosen the wavenumber to be positive.

\subsubsection{Asymptotic Analysis} \label{sec:AA_2}

Let us consider $\w$ near $\w_{\d}$, and define the coefficients
\begin{align*}
c_n=
    \begin{cases}
        \log(\d k_0 \hat{\g}), \quad &n=-1,\\
        1, &n=0,\\
        (\d k_0)^{2n}\log(\d k_0 \hat{\g}), &n\geq 1.
    \end{cases}
\end{align*}
Then, we can write
$$
K^{\d k_0}_D = \sum_{n=-1}^{+\infty} c_n K^{(n)}_D.
$$
We are interested in studying the singularities of the operator $(I - \d^2 \w^2 \xi(\w,k) K^{\d k_0}_D)^{-1}$. 
Setting $B_n := \d^2 \w^2 \xi(\w,k) K^{(n)}_D$, we find that $(I - \d^2 \w^2 \xi(\w,k) K^{\d k_0}_D)^{-1} = 0$ can be written as
\begin{align*}
    \left(I - \sum_{n=-1}^{+\infty} c_n B_n \right)^{-1} = 0,
\end{align*}
which can be expanded to give
\begin{align*}
    &\left( I - \log(\d k_0 \hat{\g})B_{-1} - B_0 - (\d k_0)^{2}\log(\d k_0 \hat{\g}) B_1 - \sum_{n\geq2}^{+\infty}c_n B_n \right)^{-1}=0,
\end{align*}
which gives
\begin{align*}
   \mathfrak{L} + \mathfrak{L} (\d k_0)^{4}\log(\d k_0 \hat{\g}) B_2 \mathfrak{L} + O(\d^6) =0,
\end{align*}
where we have defined,
\begin{equation}\label{L}
\mathfrak{L} := \Big( I - \log(\d k_0 \hat{\g})B_{-1} - B_0 - (\d k_0)^{2}\log(\d k_0 \hat{\g}) B_1 \Big)^{-1}  .
\end{equation}
Using this expression, we have the following proposition.
\begin{proposition}
Let $d=2$ and let $\w_{\d}$ be defined by (\ref{w_e k_e for d=2}). Then, as $\delta\to0$, the $O(\d^4)$ approximations of the subwavelength resonant frequencies $\w_s$ and the associated wavenumbers $k_s$ of the single halide perovskite resonator $\Omega = \d D + z$ satisfy
\begin{align}\label{Headshot}
    1 - \d^2\w_s^2\xi(\w_s,k_s)\l_{\d} = - \d^6 k_0^4 \log(\d k_0 \hat{\g}) \w_s^2 \xi(\w_s,k_s) \langle K^{(2)}_D[u_{\d}], u_{\d} \rangle.
\end{align}
\end{proposition}
\begin{proof}
Applying a pole-pencil decomposition, we obtain
$$
 \Big( I - \log(\d k_0 \hat{\g})B_{-1} - B_0 - (\d k_0)^{2}\log(\d k_0 \hat{\g}) B_1 \Big)^{-1}[\cdot] = \frac{\langle \cdot,u_{\d}\rangle u_{\d} }{1 - \d^2\w_s^2\xi(\w_s,k_s)\l_{\d}} + R(\w_s)[\cdot],
$$
where the remainder $R(\w_s)$ is analytic in a neighborhood of $\w_{\d}$ and can be dropped. Thus, dropping the $O(\d^6)$ term, we find for $\psi\in L^2(D)$ that
\begin{align*}
    \Big[\mathfrak{L} + \mathfrak{L} (\d k_0)^{4}\log(\d &k_0 \hat{\g}) B_2 \mathfrak{L}\Big](\psi) = 0.
\end{align*}
Applying a pole-pencil decomposition on \eqref{L}, we get
\begin{align*}
    \frac{\langle \psi,u_{\d}\rangle u_{\d}}{1 - \d^2\w_s^2\xi(\w_s,k_s)\l_{\d}} + \frac{\langle \psi,u_{\d}\rangle u_{\d} }{1 - \d^2\w_s^2\xi(\w_s,k_s)\l_{\d}} &(\d k_0)^{4}\log(\d k_0 \hat{\g}) B_2 \frac{\langle \psi,u_{\d}\rangle u_{\d}}{1 - \d^2\w_s^2\xi(\w_s,k_s)\l_{\d}} =0.
\end{align*}
This implies that
\begin{align*}
    1 - \d^2\w_s^2\xi(\w_s,k_s)\l_{\d} = - \d^6 k_0^4 &\log(\d k_0 \hat{\g}) \w_s^2 \xi(\w_s,k_s) \langle K^{(2)}_D[u_{\d}], u_{\d} \rangle,
\end{align*}
which is the desired result.
\end{proof}
In order to obtain the associated consequences of this proposition, we observe that, since $d=2$, we have that
\begin{align*}
    K^{(2)}_D[u_{\d}](x) &= \int_D \frac{\partial^2}{\partial k^2} G(x-y,k)\Big|_{k=0}u_{\d}(y)\upd y = \int_D \frac{\partial^2}{\partial k^2} \left( -\frac{i}{4} H^{(1)}_0(k|x|) \right)\Big|_{k=0}\\
    &= -\frac{i}{4} \int_D - \frac{1}{\pi|x-y|^2} u_{\d}(y) \upd y = \frac{i}{4\pi} \int_D \frac{u_{\d}(y)}{|x-y|^2} \upd y.
\end{align*}
Hence,
\begin{align} \label{S}
    \langle K^{(2)}_D[u_{\d}], u_{\d} \rangle &= \int_D \left( \frac{i}{4\pi} \int_D \frac{u_{\d}(y)}{|x-y|^2} \upd y \right) \overline{u_{\d}}(x)\upd x = \frac{i}{4\pi} \int_D \int_D \frac{u_{\d}(y) \overline{u_{\d}}(x)}{|x-y|^2} \upd y \upd x =: \frac{i}{4\pi} \mathbb{S}.
\end{align}
So, we get that (\ref{Headshot}) is equivalent to
\begin{align}\label{Headshot 2}
    1 - \d^2 \w_s^2 \xi(\w_s,k_s) \l_{\d} = \frac{ - i \d^6 k_0^4 \log(\d k_0 \hat{\g}) \w_s^2 \xi(\w_s,k_s)\mathbb{S}}{4\pi}.
\end{align}
\begin{remark}
In the Appendix~\ref{app2} of this paper, we recover a formula for $\mathbb{S}$.
\end{remark}
Then, the next two corollaries follow as immediate results.
\begin{corollary}
Let $d=2$. Then, as $\delta\to0$, the $O(\d^{10}\log(\d)^3)$ approximation of the subwavelength resonant frequencies of the halide perovskite resonator $\Omega = \d D + z$ can be computed as
\begin{align}
    1 - \frac{\w_s^2}{\w_{\d}^2} \cdot \frac{16\pi^2\g^2\l_{\d}\Big(\ve(\w_s,k_s)-\ve_0\Big)}{\a^2\d^6\m_0 k_0^4 \log(\d k_0 \hat{\g})^2 \mathbb{G}^2} = - \frac{  i \d^6 k_0^4 \log(\d k_0 \hat{\g}) \w_s^2 \xi(\w_s,k_s)\mathbb{S}}{4\pi}.
\end{align}
\end{corollary}
\begin{proof}
By a direct computation, we observe that
\begin{align*}
    1 - \d^2 \w_s^2 \xi(\w_s,k_s) \l_{\d} &\ = \ 1 - \d^2 \w_s^2 \m_0 \Big(\ve(\w_s,k_s)-\ve_0\Big) \l_{\d} \frac{\w_{\d}^2}{\w_{\d}^2}\\
    &\overset{(\ref{w_e k_e for d=2})}{=} \ 1 - \d^2 \frac{\w_s^2}{\w_{\d}^2} \m_0 \Big(\ve(\w_s,k_s)-\ve_0\Big) \l_{\d} 
    \frac{16\pi^2\g^2}{\a^2\d^8\m_0^2 k_0^4 \log(\d k_0 \hat{\g})^2 \mathbb{G}^2}\\
    & \ = \ 1 - \frac{\w_s^2}{\w_{\d}^2} \cdot \frac{16\pi^2\g^2\l_{\d}\Big(\ve(\w_s,k_s)-\ve_0\Big)}{\a^2\d^6\m_0 k_0^4 \log(\d k_0 \hat{\g})^2 \mathbb{G}^2},
\end{align*}
and thus, (\ref{Headshot 2}) gives the desired result.
\end{proof}
\begin{corollary}
Let $d=2$. Then, as $\delta\to0$, the $O(\d^4 \log(\d))$ approximation of the subwavelength resonant frequencies of the halide perovskite resonator $\Omega = \d D + z$ can be computed as
\begin{align}
    1 - \d^2 \w_s^2 \xi(\w_s,k_s) \l_{\d} = - \frac{\w_s^2}{\w^2_{\ve}} \cdot \frac{ 4 \pi i \mathbb{S} \g^2 \Big( \ve(\w_s,k_s) - \ve_0 \Big) }{ \a^2 \d^2 \log(\d k_0 \hat{\g}) \m_0 \mathbb{G}^2 }.
\end{align}
\end{corollary}
\begin{proof}
Again, to show this, we need to make a straightforward calculation:
\begin{align*}
    \frac{  i \d^6 k_0^4 \log(\d k_0 \hat{\g}) \w_s^2 \xi(\w_s,k_s)\mathbb{S}}{4\pi} &\ = \ \frac{i}{4\pi} \ \d^6 \ k_0^4 \ \log(\d k_0 \hat{\g}) \ \frac{\w^2_s}{\w_{\d}^2} \ \mathbb{S} \ \m_0 \ \Big( \ve(\w_s,k_s) - \ve_0 \Big) \ \w_{\d}^2 \\
    &\overset{(\ref{w_e k_e for d=2})}{=} \ \frac{i}{4\pi} \ \d^6 \ k_0^4 \ \log(\d k_0 \hat{\g}) \ \frac{\w^2_s}{\w_{\d}^2} \ \mathbb{S} \ \m_0 \ \Big( \ve(\w_s,k_s) - \ve_0 \Big) \  \frac{16\pi^2\g^2}{\a^2\d^8\m_0^2 k_0^4 \log(\d k_0 \hat{\g})^2 \mathbb{G}^2}\\
    & \ = \ \frac{\w_s^2}{\w^2_{\ve}} \cdot \frac{ 4 \pi i \mathbb{S} \g^2 \Big( \ve(\w_s,k_s) - \ve_0 \Big) }{ \a^2 \d^2 \log(\d k_0 \hat{\g}) \m_0 \mathbb{G}^2 }.
\end{align*}
Hence, (\ref{Headshot 2}) gives the desired result.
\end{proof}
We continue our analysis in the same way as in the previous case.
\begin{proposition}
Let $d=2$. For $\w$ real close to the resonant frequency $\w_s$, the following approximation for the field scattered by the halide perovskite nano-particle $\Omega = \d D + z$ holds:
\begin{align}
    u(x) - u_{in}(x) \approx \frac{ 4\pi \Big( 1 + \d^2 \w^2 \xi(\w,k)G(x-y,\d k_0) \Big) }{i \d^6 k_0^4 \w^2 \log(\d k_0 \hat{\g}) \w^2 \xi(\w,k)  \mathbb{S}} \langle u_{in}, u_{\d} \rangle \int_D u_{\d}.
\end{align}
\begin{proof}
As we mentioned at the beginning or our analysis, our goal is to find $u\in L^{2}(D)$, $u\ne 0$, such that (\ref{pb}) is satisfied. Using the pole-pencil decomposition on this Lippmann-Schwinger formulation of the problem, we can rewrite, as in Corollary 2.3 in \cite{ADFMS},
\begin{align*}
    u(x)-u_{in}(x) \ \approx -\d^2 \w^2 \xi(\w,k) G(x-&y,\d k_0) \frac{\langle u_{in}, u_{\d} \rangle \int_D u_{\d}}{1 - \d^2 \w^2 \xi(\w,k) \l_{\d}}\\
     &+ \d^6 k_0^4 \log(\d k_0 \hat{\g}) \w^2 \xi(\w,k) \frac{ \langle K^{(2)}_D[u_{\d}], u_{\d} \rangle \langle u_{in}, u_{\d} \rangle \int_D u_{\d}}{\Big(1 - \d^2 \w^2 \xi(\w,k) \l_{\d}\Big)^2},
\end{align*}
which, using \eqref{Headshot 2}, becomes:
\begin{align*}
    u(x)-u_{in}(x) & \ \approx \ \frac{G(x-y,\d k_0)\langle u_{in}, u_{\d} \rangle \int_D u_{\d} }{\d^4 k_0^4 \log(\d k_0 \hat{\g}) \langle K^{(2)}_D[u_{\d}], u_{\d} \rangle} + \frac{\langle u_{in}, u_{\d} \rangle \int_D u_{\d}}{\d^6 k_0^4 \w^2 \log(\d k_0 \hat{\g}) \w^2 \xi(\w,k)\langle K^{(2)}_D[u_{\d}], u_{\d} \rangle}.
\end{align*}
From \eqref{S}, this gives
\begin{align*}
    u(x)-u_{in}(x) & \ \approx \ \frac{ 4\pi \Big( 1 + \d^2 \w^2 \xi(\w,k)G(x-y,\d k_0) \Big) }{i \d^6 k_0^4 \w^2 \log(\d k_0 \hat{\g}) \w^2 \xi(\w,k)  \mathbb{S}} \langle u_{in}, u_{\d} \rangle \int_D u_{\d},
\end{align*}
which is the desired result.
\end{proof}
\end{proposition}

We finish our analysis with the following result.

\begin{proposition}
Let $d=2$ and $\d$ be small enough. Then, the $o(\d^4)$ approximation of the subwavelength resonant frequencies $\w_s$ of the halide perovskite nano-particle $\Omega = \d D + z$ satisfies
\begin{align}
    1 - \d^2 \w_s^2 \xi(\w_s,k_s) \left( -\frac{|D|}{2\pi} \log(\d k_0 \hat{\g}) + \langle K^{(0)}_D[\hat{\mathbb{I}}_D], \hat{\mathbb{I}}_D \rangle + \d^2 k_0^2 \log(\d) \langle K^{(1)}_D[\hat{\mathbb{I}}_D],\hat{\mathbb{I}}_D \rangle \Big) \right) = O\big( \d^4 \big),
\end{align}
where $|D|$ is the volume of $D$ and $\hat{\mathbb{I}}_D = \mathbb{I}_D / \sqrt{|D|}$. 
\end{proposition}
\begin{proof}
We want to find $\w_s \in \mathbb{C}$ such that
\begin{align*}
    &\Big( I - \d^2 \w_s^2\xi(\w_s,k_s) K^{\d k_0}_D \Big)[u](x) = 0,
\end{align*}
which, for small $\delta$, can be written as
\begin{align*}
    \Big( I - \d^2 \w_s^2\xi(\w_s,k_s) \Big(&\log(\d k_0 \hat{\g}) K^{(-1)}_D + K^{(0)}_D + (\d k_0)^2 \log(\d k_0 \hat{\g}) K^{(1)}_D\Big) \Big)[u](x) = O\Big(\d^6 \log(\d)\Big).
\end{align*}
Let us denote
$$
M^{\d k_0}_D := \log(\d k_0 \hat{\g}) K^{(-1)}_D + K^{(0)}_D + (\d k_0)^2 \log(\d k_0 \hat{\g}) K^{(1)}_D.
$$
We take $\n(\d) \in \sigma(M^{\d k_0}_D)$ and consider the eigenvalue problem for $M^{\d k_0}_D$:
\begin{align}\label{eig M}
    M^{\d k_0}_D[\Psi] = \n(\d) \Psi.
\end{align}
We employ the ansatz
\begin{align*}
    \Psi(\d) &= \Psi_0 + O\left(\frac{1}{\log(\d)}\right),\\
    \n(\d) &= \log(\d ) \n_0 + \n_1 + \d^2 \log(\d) \n_2 + O\Big( \d^2 \Big).
\end{align*}
%
From \eqref{eig M} and the fact that $\d^2\log(\d k_0 \hat{\g})=\d^2\log(\d )+O(\d^2)$ we have that
\begin{align*}
    \Big( \log(\d) K^{(-1)}_D + \log(k_0\hat{\g})K^{(-1)}_D + K^{(0)}_D &+ (\d k_0)^2 \log(\d) K^{(1)}_D \Big) [\Psi_0] = \\
    &\Big(  \log(\d ) \n_0 + \n_1 + \d^2 \log(\d) \n_2 \Big)[\Psi_0] + O\Big( \d^2 \Big).
\end{align*}
Equating the $O(\log \d)$ terms gives
\begin{align*}
    K^{(-1)}_D[\Psi_0] = \n_0 \Psi_0 &\Rightarrow \n_0 \hat{\mathbb{I}}_D = K^{(-1)}_D [\hat{\mathbb{I}}_D] \Rightarrow \n_0 \hat{\mathbb{I}}_D = - \frac{|D|}{2\pi} \hat{\mathbb{I}}_D \Rightarrow \n_0 = - \frac{|D|}{2\pi}.
\end{align*}
Then, equating the $O(1)$ terms gives
\begin{align*}
    \log(k_0 \hat{\g}) K^{(-1)}_D[\hat{\mathbb{I}}_D] &+ K^{(0)}_D[\hat{\mathbb{I}}_D] = \n_1 \hat{\mathbb{I}}_D \Rightarrow \n_1 \hat{\mathbb{I}}_D = - \frac{|D|}{2\pi} \log(k_0 \hat{\g}) \hat{\mathbb{I}}_D + K^{(0)}_D[\hat{\mathbb{I}}_D] \\
    &\Rightarrow \n_1 = - \frac{|D|}{2\pi} \log(k_0 \hat{\g}) + \langle K^{(0)}_D[\hat{\mathbb{I}}_D], \hat{\mathbb{I}}_D \rangle.
\end{align*}
Using the same reasoning for the $O\Big(\d^2 \log(\d)\Big)$ terms, we get
\begin{align*}
    \n_2 \hat{\mathbb{I}}_D = k_0^2 K^{(1)}_D[\hat{\mathbb{I}}_D] \Rightarrow \n_2 = k_0^2 \langle K^{(1)}_D[\hat{\mathbb{I}}_D],\hat{\mathbb{I}}_D \rangle.
\end{align*}
Thus, 
\begin{align*}
    \n(\d) &= \log(\d) \n_0 + \n_1 + \d^2 \log(\d) \n_2 + O\Big( \d^2 \Big)\\
    &= -\frac{|D|}{2\pi} \log(\d) - \frac{|D|}{2\pi} \log(k_0 \hat{\g}) + \langle K^{(0)}_D[\hat{\mathbb{I}}_D], \hat{\mathbb{I}}_D \rangle + \d^2 k_0^2 \log(\d) \langle K^{(1)}_D[\hat{\mathbb{I}}_D],\hat{\mathbb{I}}_D \rangle + O\Big( \d^2 ) \Big)\\
    &= -\frac{|D|}{2\pi} \log(\d k_0 \hat{\g}) + \langle K^{(0)}_D[\hat{\mathbb{I}}_D], \hat{\mathbb{I}}_D \rangle + \d^2 k_0^2 \log(\d) \langle K^{(1)}_D[\hat{\mathbb{I}}_D],\hat{\mathbb{I}}_D \rangle + O\Big( \d^2 \Big).
\end{align*}
Using these expressions, $1 - \d^2 \w_s^2 \xi(\w_s,k_s) M^{\d k_0}_D = O ( \d^6 \log(\d) )$ can be rewritten as
\begin{align*}
    &1 - \d^2 \w_s^2 \xi(\w_s,k_s) \Big( -\frac{|D|}{2\pi} \log(\d k_0 \hat{\g}) + \langle K^{(0)}_D[\hat{\mathbb{I}}_D], \hat{\mathbb{I}}_D \rangle + \d^2 k_0^2 \log(\d) \langle K^{(1)}_D[\hat{\mathbb{I}}_D],\hat{\mathbb{I}}_D \rangle\\
    &\hspace{11cm} + O\Big( \d^2 \Big) \Big) = O\Big( \d^6 \log(\d) \Big),
\end{align*}
from which the result follows.
\end{proof}

\section{Hybridization of Two Resonators} \label{sec:two}

\subsection{Three Dimensions}

Let us consider two identical halide perovskite resontators $D_1$ and $D_2$ (e.g. the speres in Figure~\ref{fig:dimer}), made from the same material. From now on, we will denote the permittivity by $\xi(\w,k)$, where $\w$ is the frequency and $k$ the associated wavenumber. In order to generalize our results, we will define it by
\begin{align} \label{new permittivity}
    \xi(\w,k) := \frac{\m_0\a}{\b - \w^2 + \h k^2 - i \g \w},
\end{align}
where the positive constants $\a,\b,\g$ and $\h$ characterise the material.

\begin{figure}
\begin{center}
\includegraphics[scale=0.8]{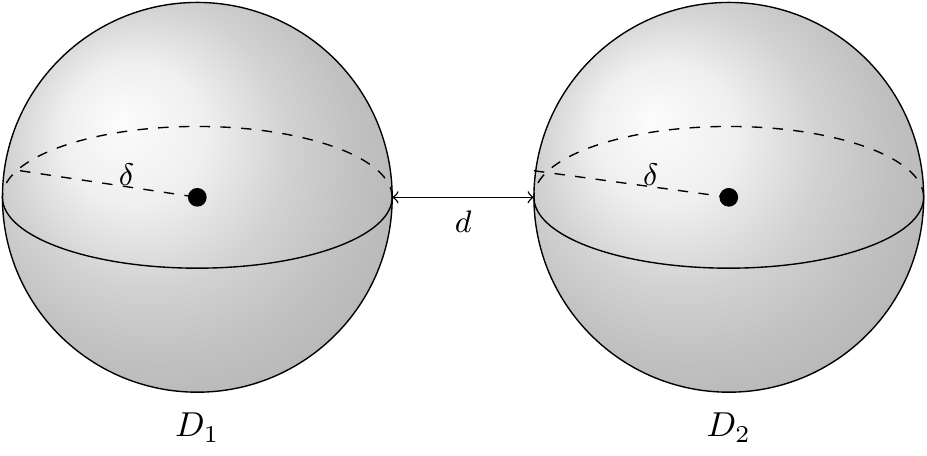}  
\end{center}
\caption{Two identical spherical halide perovskite resontators $D_1$ and $D_2$ of radius $\d$, made from the same material, placed at a distance $d$ from each other. } \label{fig:dimer}
\end{figure}

Then, since there is an interaction between the two resonators, the field $u-u_{in}$ scattered by the two particles will be given by the following representation formula:
\begin{align}\label{repping}
    (u-u_{in})(x) = - \d^2 \w^2 \xi(\w,k) \left[ \int_{D_1} G(x-y,\d k_0) u(y) \upd y + \int_{D_2} G(x-y,\d k_0) u(y) \upd y  \right] \ \ \ \text{ for } x \in \mathbb{R}^d.
\end{align}
\begin{definition} \label{def:KR}
We define the integral operators $K^{\d k_0}_{D_i}$ and $R^{\d k_0}_{D_i D_j}$, for $i,j=1,2$, by
\begin{align*}
    K^{\d k_0}_{D_i}: u\big|_{D_i} \in L^2(D_i) \longmapsto - \int_{D_i} G(x-y,\d k_0) u(y) \upd y \Big|_{D_i} \in L^2(D_i)
\end{align*}
and
\begin{align*}
    R^{\d k_0}_{D_i D_j}: u\big|_{D_i} \in L^2(D_i) \longmapsto - \int_{D_i} G(x-y,\d k_0) u(y) \upd y \Big|_{D_j} \in L^2(D_j).
\end{align*}
\end{definition}
Then, the following lemma is a direct consequence of these definitions.
\begin{lemma}
The scattering problem (\ref{repping}) can be restated, using the Definition~\ref{def:KR}, as
\begin{align}\label{system 1}
    \begin{pmatrix}
    1 - \d^2 \w^2 \xi(\w,k) K^{\d k_0}_{D_1} & - \d^2 \w^2 \xi(\w,k) R^{\d k_0}_{D_2 D_1} \\ 
    -\d^2 \w^2 \xi(\w,k) R^{\d k_0}_{D_1 D_2} & 1 - \d^2 \w^2 \xi(\w,k) K^{\d k_0}_{D_2}
    \end{pmatrix}
    \begin{pmatrix}
    u|_{D_1}\\
    u|_{D_2}
    \end{pmatrix}
    =
    \begin{pmatrix}
    u_{in}|_{D_1}\\
    u_{in}|_{D_2}
    \end{pmatrix}.
\end{align}
\end{lemma}
Thus, the scattering resonance problem is to find $\w$ such that the operator in (\ref{system 1}) is singular, or equivalently, such that there exists $(u_1,u_2) \in L^2(D_1) \times L^2(D_2)$, $(u_1,u_2) \ne 0$, such that
\begin{align}\label{system 2}
    \begin{pmatrix}
    1 - \d^2 \w^2 \xi(\w,k) K^{\d k_0}_{D_1} & - \d^2 \w^2 \xi(\w,k) R^{\d k_0}_{D_2 D_1} \\ 
    -\d^2 \w^2 \xi(\w,k) R^{\d k_0}_{D_1 D_2} & 1 - \d^2 \w^2 \xi(\w,k) K^{\d k_0}_{D_2}
    \end{pmatrix}
    \begin{pmatrix}
    u_{1}\\
    u_{2}
    \end{pmatrix}
    =
    \begin{pmatrix}
    0\\
    0
    \end{pmatrix}.
\end{align}
\begin{theorem} \label{thm:33}
Let $d=3$. Then, the hybridized subwavelength resonant frequencies $\w$ satisfy
\begin{equation} \label{hybrid 3}
    \Big(1 - \d^2 \w^2\xi(\w,k)\l_{\d}\Big)^2 - \d^4 \w^4 \xi(\w,k)^2 \langle R^{\d k_0}_{D_1 D_2} \phi_1^{(\d)}, \phi_2^{(\d)} \rangle \langle R^{\d k_0}_{D_2 D_1} \phi_2^{(\d)}, \phi_1^{(\d)} \rangle = 0,
\end{equation}
where $\phi_i^{(\d)}$, for $i=1,2$, is the eigenfunction associated to the eigenvalue $\l_{\d}$ of the potential $K^{\d k_0}_{D_i}$.
\end{theorem}
\begin{proof}
We observe that (\ref{system 2}) is equivalent to
\begin{align*}
    \begin{pmatrix}
    1 - \d^2 \w^2 \xi(\w,k) K^{\d k_0}_{D_1} & 0 \\ 
    0 & 1 - \d^2 \w^2 \xi(\w,k) K^{\d k_0}_{D_2}
    \end{pmatrix}
    \begin{pmatrix}
    u_1\\
    u_2
    \end{pmatrix}
    - \d^2 \w^2 \xi(\w,k)
    \begin{pmatrix}
    0 & R^{\d k_0}_{D_2 D_1} \\ 
    R^{\d k_0}_{D_1 D_2} & 0
    \end{pmatrix}
    \begin{pmatrix}
    u_1 \\
    u_2
    \end{pmatrix} = 0,
\end{align*}
which gives 
\begin{align}\label{system 3}
    \begin{pmatrix}
    u_1 \\
    u_2
    \end{pmatrix}
    - \d^2 \w^2 \xi(\w,k)
    \begin{pmatrix}
    \Big(1 - \d^2 \w^2 \xi(\w,k) K^{\d k_0}_{D_1}\Big)^{-1} & 0 \\ 
    0 & \Big(1 - \d^2 \w^2 \xi(\w,k) K^{\d k_0}_{D_2}\Big)^{-1}
    \end{pmatrix}
    \begin{pmatrix}
    R^{\d k_0}_{D_2 D_1} u_2 \\
    R^{\d k_0}_{D_1 D_2} u_1
    \end{pmatrix}
    =0.
\end{align}
Let us now apply a pole-pencil decomposition on the operators $\Big(1 - \d^2 \w^2 \xi(\w,k) K^{\d k_0}_{D_i}\Big)^{-1}$, for $i=1,2$. We see that
\begin{align*}
    \Big(1 - \d^2 \w^2 \xi(\w,k) K^{\d k_0}_{D_1}\Big)^{-1}(\cdot) = \frac{ \langle \cdot, \phi_1^{(\d)} \rangle \phi_1^{(\d)} }{1 - \d^2 \w^2 \xi(\w,k) \l_{\d}} + R_1[\w](\cdot)
\end{align*}
and
\begin{align*}
    \Big(1 - \d^2 \w^2 \xi(\w,k) K^{\d k_0}_{D_2}\Big)^{-1}(\cdot) = \frac{ \langle \cdot, \phi_2^{(\d)} \rangle \phi_2^{(\d)} }{1 - \d^2 \w^2 \xi(\w,k) \l_{\d}} + R_2[\w](\cdot),
\end{align*}
where the remainder terms $R_1[\w](\cdot)$ and $R_2[\w](\cdot)$ are holomorphic for $\w$ in a neighborhood of $\w_{\d}$, so can be neglected. Then, (\ref{system 3}), is equivalent to
\begin{align*}
    \begin{pmatrix}
    u_1 \\
    u_2
    \end{pmatrix}
    - \d^2 \w^2 \xi(\w,k)
    \begin{pmatrix}
    \frac{ \langle \cdot, \phi_1^{(\d)} \rangle \phi_1^{(\d)} }{1 - \d^2 \w^2 \xi(\w,k) \l_{\d}} & 0 \\ 
    0 & \frac{ \langle \cdot, \phi_2^{(\d)} \rangle \phi_2^{(\d)} }{1 - \d^2 \w^2 \xi(\w,k) \l_{\d}}
    \end{pmatrix}
    \begin{pmatrix}
    R^{\d k_0}_{D_2 D_1} u_2 \\
    R^{\d k_0}_{D_1 D_2} u_1
    \end{pmatrix}
    =0.
\end{align*}
This gives us the system
\begin{align*}
    \begin{cases}
        u_1 - \frac{\d^2 \w^2 \xi(\w,k)}{1 - \d^2 \w^2 \xi(\w,k) \l_{\d}} \langle R^{\d k_0}_{D_2 D_1} u_2, \phi_1^{(\d)} \rangle \phi_1^{(\d)}=0,\\
        u_2 - \frac{\d^2 \w^2 \xi(\w,k)}{1 - \d^2 \w^2 \xi(\w,k) \l_{\d}} \langle R^{\d k_0}_{D_1 D_2} u_1, \phi_2^{(\d)} \rangle \phi_2^{(\d)}=0.
    \end{cases}
\end{align*}
Applying the operator $R^{\d k_0}_{D_1 D_2}$ (resp. $R^{\d k_0}_{D_2 D_1}$) to the first (resp. second) equation, and then applying $\langle \cdot, \phi_2^{(\ve)} \rangle$ (resp. $\langle \cdot, \phi_1^{(\ve)} \rangle$), we find that
\begin{align*}
    \begin{cases}
        \langle R^{\d k_0}_{D_1 D_2} u_1, \phi_2^{(\d)} \rangle - \frac{\d^2 \w^2 \xi(\w,k)}{1 - \d^2 \w^2 \xi(\w,k) \l_{\d}} \langle R^{\d k_0}_{D_1 D_2} \phi_1^{(\d)}, \phi_2^{(\d)} \rangle \langle R^{\d k_0}_{D_2 D_1} u_2, \phi_1^{(\d)} \rangle =0,\\
        \langle R^{\d k_0}_{D_2 D_1} u_2, \phi_1^{(\d)} \rangle - \frac{\d^2 \w^2 \xi(\w,k)}{1 - \d^2 \w^2 \xi(\w,k) \l_{\d}} \langle R^{\d k_0}_{D_2 D_1} \phi_2^{(\d)}, \phi_1^{(\d)} \rangle \langle R^{\d k_0}_{D_1 D_2} u_1, \phi_2^{(\d)} \rangle =0.
    \end{cases}
\end{align*}
This system has a solution only if its determinant is zero. That is, if
\begin{align*}
    1 - \frac{ \d^4 \w^4 \xi(\w,k)^2 }{\Big( 1 - \d^2 \w^2 \xi(\w,k) \l_{\d} \Big)^2} \langle R^{\d k_0}_{D_1 D_2} \phi_1^{(\d)}, \phi_2^{(\d)} \rangle \langle R^{\d k_0}_{D_2 D_1} \phi_2^{(\d)}, \phi_1^{(\d)} \rangle = 0,
\end{align*}
which gives the desired result.
\end{proof}
The following corollary is a direct result of Theorem~\ref{thm:33}.
\begin{corollary}
Let $d=3$. Then, the hybridized subwavelength resonant frequencies are given by
\begin{gather}
    \w = \frac{i\g \pm \sqrt{-\g^2 - 4 \Gamma(\b + \h k^2)}}{2 \Gamma}, \label{hybr. freq. d=3} \\ \text{where}\qquad\Gamma=-1-\d^2\a\l_{\d} \pm \a \d^2 \sqrt{\langle R^{\d k_0}_{D_1 D_2} \phi_1^{(\d)}, \phi_2^{(\d)} \rangle \langle R^{\d k_0}_{D_2 D_1} \phi_2^{(\d)}, \phi_1^{(\d)} \rangle}. \nonumber
\end{gather}
where $\phi_i^{\d}$, for $i=1,2$, is the eigenfunction associated to the eigenvalue $\l_{\d}$ of the potential $K^{\d k_0}_{D_i}$ and the $\pm$ in the two expressions do not have to agree.
\end{corollary}
\begin{proof}
We introduce the notation $\mathbb{K} := \langle R^{\d k_0}_{D_1 D_2} \phi_1^{(\d)}, \phi_2^{(\d)} \rangle$ and $\mathbb{M} := \langle R^{\d k_0}_{D_2 D_1} \phi_2^{(\d)}, \phi_1^{(\d)} \rangle$. Then, (\ref{hybrid 3}) becomes
\begin{align*}
    \Big(1 - \d^2 \w^2\xi(\w,k)\l_{\d}&\Big)^2 - \d^4 \w^4 \xi(\w,k)^2 \mathbb{K} \mathbb{M} = 0 \  \Leftrightarrow \ 1 - \d^2 \w^2\xi(\w,k)\l_{\d} \pm \d^2 \w^2 \xi(\w,k) \sqrt{\mathbb{K}\mathbb{M}} = 0 \ \Leftrightarrow \\
    &\left( -1-\d^2\a\l_{\d} \pm \a \d^2 \sqrt{\mathbb{K}\mathbb{M}} \right) \w^2 - i\g\w + \b + \h k^2 = 0,
\end{align*}
and the roots to this second degree polynomial are given by
\begin{align*}
    \w = \frac{i\g \pm \sqrt{-\g^2 - 4 \Gamma(\b+\h k^2)}}{2 \Gamma} \quad \text{where}\quad \Gamma=-1-\d^2\a\l_{\d} \pm \a \d^2 \sqrt{\mathbb{K}\mathbb{M}},
\end{align*}
with the two $\pm$ not necessarily agreeing. Finally, substituting the expressions for $\mathbb{K}$ and $\mathbb{M}$, we obtain the result.
\end{proof}

\subsection{Two Dimensions}

Let us move on to the case of dimension $d=2$. For simplicity, we again consider two identical halide perovskite resontators $D_1$ and $D_2$, made from the same material with permittivity given by the formula (\ref{new permittivity}). We define the operators $K^{\d k_0}_{D_i}$ and $R^{\d k_0}_{D_i D_j}$, for $i,j=1,2$, as in Definition \ref{def:KR} and we continue by defining the following integral operators.
\begin{definition}
We define the integral operators $M^{\d k_0}_{D_i}$ and $N^{\d k_0}_{D_i D_j}$ for $i,j=1,2$ as
\begin{align*}
    M^{\d k_0}_{D_i} := \hat{K}^{\d k_0}_{D_i} + K^{(0)}_{D_i} + (\d k_0)^2 \log(\d k_0 \hat{\g}) K^{(1)}_{D_i},
\end{align*}
and
\begin{align*}
    N^{\d k_0}_{D_i D_j} := \hat{K}^{\d k_0}_{D_i D_j} + R^{(0)}_{D_i D_j} + (\d k_0)^2 \log(\d k_0 \hat{\g}) R^{(1)}_{D_i D_j},
\end{align*}
where
\begin{align*}
    K^{(0)}_{D_i}: u\Big|_{D_i} \in L^2(D_i) &\longmapsto \int_{D_i} G(x-y,0)u(y)\upd y \Big|_{D_i} \in L^{2}(D_i),\\
    \hat{K}^{\d k_0}_{D_i}: u\Big|_{D_i} \in L^2(D_i) &\longmapsto \log(\hat{\g} \d k_0) \hat{K}_{D_i}[u]\Big|_{D_i}\in L^{2}(D_i), \\
    \hat{K}_{D_i}: u\Big|_{D_i} \in L^2(D_i) &\longmapsto - \frac{1}{2\pi} \int_{D_i} u(y) \upd y \Big|_{D_i} \in L^{2}(D_i),\\
    K^{(1)}_{D_i}: u\Big|_{D_i} \in L^2(D_i) &\longmapsto \int_{D_i} \frac{\partial}{\partial k} G(x-y,k)\Big|_{k=0} u(y) \upd y \Big|_{D_i} \in L^2(D_i),
\end{align*}
and
\begin{align*}
    R^{(0)}_{D_i D_j}: u\Big|_{D_i} \in L^2(D_i) &\longmapsto \int_{D_i} G(x-y,0)u(y)\upd y \Big|_{D_j} \in L^{2}(D_j),\\
    \hat{K}^{\d k_0}_{D_i D_j}: u\Big|_{D_i} \in L^2(D_i) &\longmapsto \log(\hat{\g} \d k_0) \hat{K}_{D_i D_j}[u]\Big|_{D_j}\in L^{2}(D_j), \\
    \hat{K}_{D_i D_j}: u\Big|_{D_i} \in L^2(D_i) &\longmapsto - \frac{1}{2\pi} \int_{D_i} u(y) \upd y \Big|_{D_j} \in L^{2}(D_j),\\
    R^{(1)}_{D_i D_j}: u\Big|_{D_i} \in L^2(D_i) &\longmapsto \int_{D_i} \frac{\partial}{\partial k} G(x-y,k)\Big|_{k=0} u(y) \upd y \Big|_{D_j} \in L^2(D_j).
\end{align*}
\end{definition}
We observe the following result.
\begin{proposition}
For the integral operators $K^{\d k_0}_{D_i}$ and $R^{\d k_0}_{D_i D_j}$, we can write
\begin{align} \label{M}
    K^{\d k_0}_{D_i} = M^{\d k_0}_{D_i} + O\Big(\d^4 \log(\d)\Big), \quad
    \text{and} \quad
    R^{\d k_0}_{D_i D_j} = N^{\d k_0}_{D_i D_j}+ O\Big(\d^4 \log(\d)\Big),
\end{align}
as $\delta\to0$ and with $k_0$ fixed.
\end{proposition}
\begin{proof}
The proof is a direct result of the expansion of the Green's function in dimension $d=2$. Indeed, for $u|_{D_i} \in L^2(D_i)$, we observe that
\begin{align*}
    K^{\d k_0}_{D_i}[u](x) &= - \int_{D_i} G(x-y,\d k_0) u(y) \upd y \Big|_{D_i}\\
    &= -\int_{D_i} \Big( \log(\hat{\g} \d k_0)  \frac{1}{2\pi} + G(x-y,0) + (\d k_0)^2 \log(\d k_0 \hat{\g}) \frac{\partial}{\partial k} G(x-y,k)\Big|_{k=0}\\
    & \ \ \ \ \ \ \ \ \ \ \ \ \  \ + O\Big(\d^4 \log(\d)\Big) \Big) u(y) \upd y \Big|_{D_i}\\
    &= \Big(\hat{K}^{\d k_0}_{D_i} + K^{(0)}_{D_i} + (\d k_0)^2 \log(\d k_0 \hat{\g}) K^{(1)}_{D_i}\Big)[u](x) + O\Big(\d^4 \log(\d)\Big)\\
    &= M^{\d k_0}_{D_i}[u](x) + O\Big(\d^4 \log(\d)\Big).
\end{align*}
Similarly, for $u|_{D_i} \in L^2(D_i)$,
\begin{align*}
    R^{\d k_0}_{D_i D_j}[u](x) &= - \int_{D_i} G(x-y,\d k_0) u(y) \upd y \Big|_{D_j} \in L^2(D_j)\\
    &= -\int_{D_i} \Big( \log(\hat{\g} \d k_0)  \frac{1}{2\pi} + G(x-y,0) + (\d k_0)^2 \log(\d k_0 \hat{\g}) \frac{\partial}{\partial k} G(x-y,k)\Big|_{k=0}\\
    & \ \ \ \ \ \ \ \ \ \ \ \ \  \ + O\Big(\d^4 \log(\d)\Big) \Big) u(y) \upd y \Big|_{D_j}\\
    &= \Big( \hat{K}^{\d k_0}_{D_i D_j} + R^{(0)}_{D_i D_j} + (\d k_0)^2 \log(\d k_0 \hat{\g}) R^{(1)}_{D_i D_j} \Big)[u](x) + O\Big(\d^4 \log(\d)\Big)\\
    &= N^{\d k_0}_{D_i D_j}[u](x)+ O(\d^4 \log(\d)).
\end{align*}
\end{proof}
Therefore, our problem is to determine the frequencies $\w$ and the associated wavenumber $k$, for which the following holds:
\begin{align} \label{Chomp}
    \begin{pmatrix}
    I - \d^2 \w^2 \xi(\w,k)K^{\d k_0}_{D_1}  & - \d^2 \w^2 \xi(\w,k) R^{\d k_0}_{D_2 D_1}\\
    - \d^2 \w^2 \xi(\w,k) R^{\d k_0}_{D_1 D_2} & I - \d^2 \w^2 \xi(\w,k)K^{\d k_0}_{D_2}
    \end{pmatrix}
    \begin{pmatrix}
    u_1 \\
    u_2
    \end{pmatrix}
    =
    \begin{pmatrix}
    0\\
    0
    \end{pmatrix}
\end{align}
for nontrivial $u:=(u_1,u_2)$, such that $u|_{D_i} \in L^2(D_i)$, for $i=1,2$.
\begin{proposition}
Let $d=2$. Then, the hybridized subwavelength resonant frequencies $\w$ satisfy
\begin{align} \label{hybrid 2}
\begin{split}
    1 - \d^2 \w^2 \xi(\w,k) &\Big( -\frac{|D_1|}{2\pi} \log(\hat{\g} \d k_0) (1 \pm 1) + \langle K^{(0)}_{D_1}[\hat{\mathbb{I}}_{D_1}], \hat{\mathbb{I}}_{D_1} \rangle + (\d k_0)^2 \log(\d k_0 \hat{\g}) \langle K^{(1)}_{D_1}[\hat{\mathbb{I}}_{D_1}], \hat{\mathbb{I}}_{D_1} \rangle \\
    &\pm \langle R^{(0)}_{D_2 D_1}[\hat{\mathbb{I}}_{D_2}], \hat{\mathbb{I}}_{D_1} \rangle \pm (\d k_0)^2 \log(\d k_0 \hat{\g}) \langle R^{(1)}_{D_2 D_1}[\hat{\mathbb{I}}_{D_2}], \hat{\mathbb{I}}_{D_1} \rangle \Big) = 0,
\end{split}
\end{align}
where the $\pm$ symbols coincide.
\end{proposition}
\begin{proof}
The first thing that we do is to observe that, by applying the expansion (\ref{M}) to (\ref{Chomp}), we reach the problem
\begin{align*}
     \begin{pmatrix}
    I - \d^2 \w^2 \xi(\w,k) M^{\d k_0}_{D_1}  & - \d^2 \w^2 \xi(\w,k) N^{\d k_0}_{D_2 D_1}\\
    - \d^2 \w^2 \xi(\w,k) N^{\d k_0}_{D_1 D_2} & I - \d^2 \w^2 \xi(\w,k) M^{\d k_0}_{D_2}
    \end{pmatrix}
    \begin{pmatrix}
    u_1 \\
    u_2
    \end{pmatrix}
    =
    \begin{pmatrix}
    O\Big(\d^4 \log(\d)\Big)\\
    O\Big(\d^4 \log(\d)\Big)
    \end{pmatrix}.
\end{align*}
We note that $|D_1| = |D_2|$. Then, using the symmetries of the dimer, let us denote
\begin{align*}
    \hat{\n}(\d) :&= - \frac{|D_1|}{2\pi} \log(\d k_0 \hat{\g}) + \langle K^{(0)}_{D_1}[\hat{\mathbb{I}}_{D_1}], \hat{\mathbb{I}}_{D_1} \rangle + (\d k_0)^2 \log(\d k_0 \hat{\g}) \langle K^{(1)}_{D_1}[\hat{\mathbb{I}}_{D_1}], \hat{\mathbb{I}}_{D_1} \rangle\\
    &= - \frac{|D_2|}{2\pi} \log(\d k_0 \hat{\g}) + \langle K^{(0)}_{D_2}[\hat{\mathbb{I}}_{D_2}], \hat{\mathbb{I}}_{D_2} \rangle + (\d k_0)^2 \log(\d k_0 \hat{\g}) \langle K^{(1)}_{D_2}[\hat{\mathbb{I}}_{D_2}], \hat{\mathbb{I}}_{D_2} \rangle,
\end{align*}
and
\begin{align*}
    \hat{\h} :&= \langle N^{\d k_0}_{D_1 D_2} [\hat{\mathbb{I}}_{D_1}], \hat{\mathbb{I}}_{D_2} \rangle
    = \langle N^{\d k_0}_{D_2 D_1} [\hat{\mathbb{I}}_{D_2}], \hat{\mathbb{I}}_{D_1} \rangle.
\end{align*}
In addition, we have that
\begin{align*}
    \hat{K}^{\d k_0}_{D_i D_j}[\hat{\mathbb{I}}_{D_i}] = \hat{K}^{\d k_0}_{ D_j}[\hat{\mathbb{I}}_{D_j}].
\end{align*}
Now, we define the quantity $\n(\d)$ to be the eigenvalues of the operator $M^{\d k_0}_{D_i}$, that is,
\begin{align*}
    \n(\d) = \langle M^{\d k_0}_{D_1}[\Psi_{D_1}], \Psi_{D_1} \rangle = \langle M^{\d k_0}_{D_2}[\Psi_{D_2}], \Psi_{D_2} \rangle,
\end{align*}
for the eigenfunctions $\Psi_{D_i}(\d) \in L^2(D_i)$, $\Psi_{D_i}(\d) = \hat{\mathbb{I}}_{D_i} + O\left( \frac{1}{\log(d)} \right)$. 
Thus, we have that (\ref{Chomp}) is equivalent to
\begin{align}\label{Chomp 2}
    \begin{pmatrix}
    u_1\\
    u_2
    \end{pmatrix}
    - \d^2 \w^2 \xi(\w,k)
    \begin{pmatrix}
    \Big( I - \d^2 \w^2 \xi(\w,k)M^{\d k_0}_{D_1} \Big)^{-1} & 0 \\
    0 & \Big( I - \d^2 \w^2 \xi(\w,k)M^{\d k_0}_{D_2} \Big)^{-1}
    \end{pmatrix}
    \begin{pmatrix}
    N^{\d k_0}_{D_2 D_1} u_2\\
    N^{\d k_0}_{D_1 D_2} u_1
    \end{pmatrix}
    =0.
\end{align}
Applying a pole-pencil decomposition, we observe that
\begin{align*}
    \Big( I - \d^2 \w^2 \xi(\w,k)M^{\d k_0}_{D_i} \Big)^{-1}[\cdot] = \frac{\langle \cdot, \hat{\mathbb{I}}_{D_i} \rangle \hat{\mathbb{I}}_{D_i}}{1 - \d^2 \w^2 \xi(\w,k) \n(\d)} + R[\w](\cdot),
\end{align*}
where the remainder terms $R[\w](\cdot)$ can be neglected. Hence, (\ref{Chomp 2}) is equivalent to
\begin{align*}
    \begin{cases}
        u_1 - \d^2 \w^2 \xi(\w,k) \frac{ \langle N^{\d k_0}_{D_2 D_1} u_2, \hat{\mathbb{I}}_{D_1} \rangle \hat{\mathbb{I}}_{D_1}}{ 1 - \d^2 \w^2 \xi(\w,k) \n(\d) } = 0, \\
        u_2 - \d^2 \w^2 \xi(\w,k) \frac{ \langle N^{\d k_0}_{D_1 D_2} u_1, \hat{\mathbb{I}}_{D_2} \rangle \hat{\mathbb{I}}_{D_2}}{ 1 - \d^2 \w^2 \xi(\w,k) \n(\d) } = 0,
    \end{cases}
\end{align*}
which is equivalent to
\begin{align*}
\begin{cases}
    \langle N^{\d k_0}_{D_1 D_2} u_1, \hat{\mathbb{I}}_{D_2} \rangle - \frac{\d^2 \w^2 \xi(\w,k)}{1 - \d^2 \w^2 \xi(\w,k)\n(\d)} \langle N^{\d k_0}_{D_1 D_2} \hat{\mathbb{I}}_{D_1}, \hat{\mathbb{I}}_{D_2} \rangle \langle N^{\d k_0}_{D_2 D_1} u_2, \hat{\mathbb{I}}_{D_1} \rangle = 0,\\
    \langle N^{\d k_0}_{D_2 D_1} u_2, \hat{\mathbb{I}}_{D_1} \rangle - \frac{\d^2 \w^2 \xi(\w,k)}{1 - \d^2 \w^2 \xi(\w,k)\n(\d)} \langle N^{\d k_0}_{D_2 D_1} \hat{\mathbb{I}}_{D_2}, \hat{\mathbb{I}}_{D_1} \rangle \langle N^{\d k_0}_{D_1 D_2} u_1, \hat{\mathbb{I}}_{D_2} \rangle = 0.
\end{cases}
\end{align*}
For this to have a solution, we need the determinant of the matrix induced by this system to be zero. 
This gives
\begin{align*}
    &1 - \frac{\d^4 \w^4 \xi(\w,k)^2}{( 1 - \d^2 \w^2 \xi(\w,k)\n(\d) )^2} \langle N^{\d k_0}_{D_1 D_2} \hat{\mathbb{I}}_{D_1}, \hat{\mathbb{I}}_{D_2} \rangle \langle N^{\d k_0}_{D_2 D_1} \hat{\mathbb{I}}_{D_2}, \hat{\mathbb{I}}_{D_1} \rangle = 0.
\end{align*}
Given the symmetry of our setting, we have that
$$
\langle N^{\d k_0}_{D_1 D_2} \hat{\mathbb{I}}_{D_1}, \hat{\mathbb{I}}_{D_2} \rangle = \langle N^{\d k_0}_{D_2 D_1} \hat{\mathbb{I}}_{D_2}, \hat{\mathbb{I}}_{D_1} \rangle,
$$
and hence, we get
\begin{align*}
    1 - \d^2 \w^2 \xi(\w,k)\n(\d) \pm \d^2 \w^2 \xi(\w,k)\langle N^{\d k_0}_{D_1 D_2} \hat{\mathbb{I}}_{D_1}, \hat{\mathbb{I}}_{D_2} \rangle = 0.
\end{align*}
This is equivalent to
\begin{align*}
    1 - \d^2 \w^2 \xi(\w,k) \Big( - \frac{|D_1|}{2\pi} \log(\d k_0 \hat{\g})(1\pm1) + &\langle K^{(0)}_{D_1}[\hat{\mathbb{I}}_{D_1}], \hat{\mathbb{I}}_{D_1} \rangle + (\d k_0)^2 \log(\d k_0 \hat{\g}) \langle K^{(1)}_{D_1}[\hat{\mathbb{I}}_{D_1}], \hat{\mathbb{I}}_{D_1} \rangle\\
    & \pm \langle R^{(0)}_{D_2 D_1} [\hat{\mathbb{I}}_{D_2}], \hat{\mathbb{I}}_{D_1} \rangle \pm (\d k_0)^2 \log(\d k_0 \hat{\g}) \langle R^{(1)}_{D_2 D_1} [\hat{\mathbb{I}}_{D_2}], \hat{\mathbb{I}}_{D_1} \rangle \Big) = 0,
\end{align*}
which is the desired result.
\end{proof}

\section{Example: Circular Resonators}

\begin{figure}
\begin{center}
\includegraphics[trim=0 0 0 1.2cm, clip,scale=0.8]{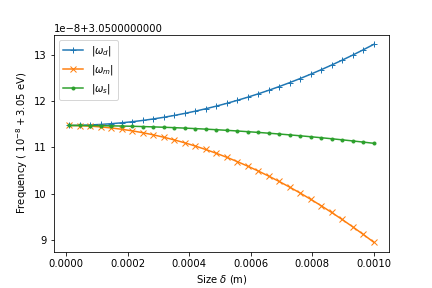}
\end{center}
\caption{Behaviour of the subwavelength resonances for small circular nano-particles of radius $\d$. The resonant frequency $\w_s$ of a single circular methylammonium lead chloride nano-particle is shown. For two circular nano-particles, made from the same material, we see how the hybridization causes the frequencies $\w_d$ (dipole) and $\w_m$ (monopole) to shift either side of $\w_s$.} \label{fig:numerics}
\end{figure}

In this section, we illustrate our results for the case of two-dimensional circular halide perovskite resonators. We can find the resonant frequencies of a single particle $\w_s$ by solving (\ref{Headshot 2}). Similarly, the hybridized resonant frequencies of a pair of circular resonators can be found by solving (\ref{hybrid 2}). The two solutions of (\ref{hybrid 2}) are denoted by $\w_m$ and $\w_d$, to describe their monopolar and dipolar characteristics. As is expected from other hybridized systems, it holds that $\w_m<\w_d$. We plot these three frequencies as a function of the particle size $\d$ in Figure~\ref{fig:numerics}. We use the parameter values from \cite{MFTHBPZK} to model resonators made from methylammonium lead chloride ($\text{MAPbCl}_3$), which is a popular halide perovskite.

The first thing we observe from Figure~\ref{fig:numerics} is that in the $\d \to 0$ limit, the frequencies coincide. This is because the nano-particles behave as isolated, identical resonators when $\d$ is very small. Then, as $\d$ increases the single-particle resonance $\w_s$ always stays between the monopole and dipole frequencies of the hybridized case. The phenomenon of the dipole frequency $\w_d$ being shifted above $\w_s$ and the monopole frequency $\w_m$ being shifted below $\w_s$ is a typical behaviour of hybridized resonator systems, see e.g. \cite{ADH}.

\section{Conclusion}

We have established a new mathematical model for halide perovskite resonators. This is a significant development of the existing theory of subwavelength resonators \cite{ADFMS, ADH}, as it generalizes the techniques to dispersive settings where the permittivity of the material depends on both the frequency and the wavenumber. Given the rapidly growing use of halide perovskites in engineering applications, this theory will have a significant impact on the design of advanced devices \cite{JKM,KMBKL}. The integral methods used here are able to describe a very broad class of resonator shapes, so are an ideal approach for studying complex geometries, such as the biomimetic eye developed by \cite{GPLLZZSQKJ}. 

\section*{Data availability}

There are no associated data, arising from this work. The only data used in this work are the material parameter values for methylammonium lead chloride, which are stated in \cite{MFTHBPZK}.

\appendix
\section{Appendix}

\subsection{Calculation of Three-dimensional Constants} \label{app1}

We derive a formula for $\mathbb{F}$, which was a crucial quantity in section~\ref{sec:AA_3}, in the case of a single three-dimensional halide perovskite resonator. We have that
$$
\langle K^{(2)}_D[u_{\d}], u_{\d} \rangle = \frac{1}{8\pi}\mathbb{F}.
$$
From (\ref{Rose in Harlem 2}), we observe that
\begin{align*}
    \d^2 \w^2 \xi(\w,k) &= \frac{8\pi}{8\pi \l_{\d}-\d^2 k_0^2 \mathbb{F}}.
\end{align*}
Also, we know that $\xi(\w,k)=\m_0(\ve(\w,k)-\ve_0)$, and we have shown that
$$
\ve(\w,k) = \ve_0 + \frac{1}{\m_0\d^2\w^2\left( \l_0 - \frac{i}{4\pi} \d k_0 \mathbb{B} \right)}.
$$
Substituting this into the above equation, we get
\begin{align*}
    \mathbb{F} = \frac{8\pi}{\d^2 k_0^2} \Big( &\l_{\d} - \l_0 + \frac{i}{4\pi} \d k_0 \mathbb{B} \Big).
\end{align*}
Therefore, we obtain that
$$
\langle K^{(2)}_D[u_{\d}], u_{\d} \rangle = \frac{1}{\d^2 k_0^2} \Big( \l_{\d} - \l_0 + \frac{i}{4\pi} \d k_0 \mathbb{B} \Big).
$$

\subsection{Calculation of Two-dimensional Constants} \label{app2}

We derive a formula for $\mathbb{S}$, which was a crucial quantity in section~\ref{sec:AA_2}, in the case of a single two-dimensional halide perovskite resonator. We have that
$$
\langle K^{(2)}_D[u_{\d}], u_{\d} \rangle = \frac{i}{4\pi}\mathbb{S}.
$$
From (\ref{Headshot 2}), we can obtain an expression for $\mathbb{S}$. Indeed, (\ref{Headshot 2}) is equivalent to
\begin{align*}
    &4\pi = \d^2 \w^2 \xi(\w,k) \Big( 4\pi \l_{\d} - i \d^4 k_0^4 \log(\d k_0 \hat{\g}) \mathbb{S} \Big).
\end{align*}
We know that $\xi(\w,k) = \m_0\Big(\ve(\w,k) - \ve_0\Big)$ and we have shown that
$$
\ve(\w,k) = \frac{1}{\m_0 \d^2 \w^2 \left(  \log(\d k_0 \hat{\g}) \l_{-1} - \frac{\mathbb{P}}{2\pi} - \frac{i(\d k_0)^2 \log(\d k_0 \hat{\g}) \mathbb{G}}{4\pi} \right)} + \ve_{0}.
$$
Substituting this into the above equality, we get
\begin{align*}
    & \ \ \ \ \ \ \ \ \ \ \ \ \mathbb{S} = \frac{-i}{\d^4 k_0^4 \log(\d k_0 \hat{\g})} \Big( 4\pi (\l_{\d} - \log(\d k_0 \hat{\g}) \l_{-1}) + 2 \mathbb{P} + i (\d k_0)^2 \log(\d k_0 \hat{\g})\mathbb{G} \Big).
\end{align*}
Therefore, we obtain that
$$
\langle K^{(2)}_D[u_{\d}], u_{\d} \rangle = \frac{1}{4\pi \d^4 k_0^4 \log(\d k_0 \hat{\g})} \Big( 4\pi (\l_{\d} - \log(\d k_0 \hat{\g}) \l_{-1}) + 2 \mathbb{P} + i (\d k_0)^2 \log(\d k_0 \hat{\g})\mathbb{G} \Big).
$$

\end{document}